\documentclass[preprint,3p]{elsarticle}

\usepackage{amssymb,amsmath,amsthm}
\usepackage{comment,enumerate,framed,hyperref}

\usepackage{tikz}
\usetikzlibrary{decorations.markings,intersections,calc,patterns}

\newcommand{\alg}{\mathbf}
\newcommand{\class}{\mathsf}
\newcommand{\set}[2]{\{ #1 \mid #2 \}}
\newcommand{\pair}[2]{\langle #1, #2 \rangle}
\newcommand{\tuple}{\overline}

\newcommand{\assign}{\mathrel{:=}}
\newcommand{\bs}{\backslash}
\newcommand{\iso}{\cong}

\newcommand{\1}{\mathsf{e}}
\newcommand{\emptyword}{\varepsilon}

\newcommand{\unit}[1]{\eta_{\scriptscriptstyle#1}}
\newcommand{\counit}[1]{\varepsilon_{\scriptscriptstyle#1}}
\newcommand{\unitmap}{\eta}
\newcommand{\counitmap}{\varepsilon}

\newcommand{\sqleq}{\sqsubseteq}

\newcommand{\nsqleq}{\not \sqsubseteq}

\newcommand{\preleq}{\preceq}
\newcommand{\leqK}{\preleq_{\class{K}}}
\newcommand{\leqL}{\preleq_{\class{L}}}

\newcommand{\bsps}{\mathop{\bs\!\!\bs}}
\newcommand{\sps}{\mathop{/\!\!/}}
\newcommand{\sigmaideal}{\sigma_{{\scriptscriptstyle\downarrow}}}

\newcommand{\Tm}{\alg{Tm}}
\newcommand{\Var}{\mathrm{Var}}
\newcommand{\qine}{\alpha}

\newcommand{\cdotgamma}{\cdot_{\gamma}}

\newcommand{\veegamma}{\vee_{\gamma}}
\newcommand{\bigveegamma}{{\textstyle\bigvee_{\gamma}}}

\newcommand{\below}{\mathop{\downarrow}}
\newcommand{\gammaid}{\gamma_{\omega}}
\newcommand{\Id}{\mathop{\mathrm{Id}_{\omega}}}
\newcommand{\IdAll}{\mathop{\mathrm{Id}}}

\newcommand{\Pomon}{\mathsf{Pomon}}
\newcommand{\PomonStar}{\Pomon_{\ast}}

\newcommand{\PomonSlStar}{\Pomon_{\ast \mathsf{s}\ell}}
\newcommand{\Slmon}{\mathsf{s}\ell\mathsf{\textrm{-}mon}}
\newcommand{\SlmonStar}{\Slmon_{\ast}}

\newcommand{\FreeSem}[1]{#1^{+}}
\newcommand{\FreeMon}[1]{#1^{*}}
\newcommand{\FreeUMon}[1]{#1^{+}}
\newcommand{\FreeGroup}[1]{\mathbf{G}(#1)}
\newcommand{\FreeRedGroup}[1]{\mathbf{G}(#1) / {\sqleq}}

\newcommand{\com}[1]{#1_{\mathsf{com}}}
\newcommand{\CFreeSem}[1]{\com{#1^{+}}}
\newcommand{\CFreeMon}[1]{\com{#1^{*}}}
\newcommand{\CFreeUMon}[1]{\com{#1^{+}}}
\newcommand{\CFreeRedMon}[1]{\com{\FreeMon{#1}}}
\newcommand{\CFreeRedUMon}[1]{\com{\FreeUMon{#1}}}

\theoremstyle{plain}
\newtheorem{theorem}{Theorem}[section]
\newtheorem{lemma}[theorem]{Lemma}
\newtheorem{proposition}[theorem]{Proposition}
\newtheorem{definition}[theorem]{Definition}

\newtheorem{fact}[theorem]{Fact}

\newtheorem*{theorem*}{Theorem}
\newtheorem*{theoremA*}{Theorem A}
\newtheorem*{theoremB*}{Theorem B}
\newtheorem*{theoremC*}{Theorem C}
\newtheorem*{theoremD*}{Theorem D}
\newtheorem*{theoremE*}{Theorem E}
\newtheorem*{theoremF*}{Theorem F}
\newtheorem*{theoremG*}{Theorem G}
\newtheorem*{conjecture*}{Conjecture}
\newtheorem*{conjectureA*}{Conjecture A}
\newtheorem*{conjectureB*}{Conjecture B}

\begin{document}

\begin{frontmatter}

\author[1]{Adam \texorpdfstring{P\v{r}enosil}{Prenosil}}
\ead[1]{adam.prenosil@vanderbilt.edu}
\address[1]{Department of Mathematics, Vanderbilt University, Nashville, TN, USA}
\title{From partially ordered monoids to partially ordered groups via~free~nuclear~preimages}

\begin{abstract}
  Two fundamental constructions operating on residuated lattices and partially ordered monoids (pomonoids) are nuclear images and conuclear images. Nuclear images allow us to construct many of the ordered algebras which arise in non-classical logic (such as pomonoids, semilattice-ordered monoids, and residuated lattices) from cancellative ordered algebras. Conuclear images then allow us to construct some of these cancellative algebras from partially ordered or lattice-ordered groups. Among other things, we show that finite (com\-mutative) integral residuated lattices are precisely the finite \mbox{nuclear} images of (com\-mutative) can\-cellative integral residuated lattices and that (com\-mutative) \mbox{integrally} closed pomonoids are precisely the nuclear \mbox{images} of subpomonoids of partially ordered (Abelian) groups. The key construction is the free nuclear preimage of a pomonoid. As a by-product of our study of free nuclear preimages, we obtain a syntactic characterization of quasivarieties of pomonoids and semilattice-ordered monoids closed under nuclear images.
\end{abstract}


\end{frontmatter}

\section{Introduction}

  A classical result of Mundici states that MV-algebras (which form the algebraic counterpart of infinite-valued {\L}ukasiewicz logic) are precisely the unit intervals of negative cones of Abelian lattice-ordered groups (Abelian $\ell$-groups)~\cite{mundici86}. That is, every MV-algebra can be constructed from some Abelian $\ell$-group in two steps: we first restrict to the negative cone of the $\ell$-group (consisting of the elements below the multiplicative unit $\1$), and then we further restrict to some interval $[u, \1]$ in this negative cone, adjusting the operations of the $\ell$-group accordingly at each step. Pursuing this direction further, Dvure\v{c}enskij and Vetterlein later described the unit intervals of (not necessarily Abelian) $\ell$-groups as the so-called pseudo MV-algebras~\cite{dvurecenskij+vetterlein01a,dvurecenskij+vetterlein01b,dvurecenskij02}.

  Restricting to the negative cone and to the unit interval are special cases of two fundamental constructions in the theory of residuated lattices in particular and partially \mbox{ordered} monoids more generally: the \emph{conuclear image} and the \emph{nuclear image}. The main goal of the present paper is to investigate which algebras arise from lattice-ordered and partially ordered groups when we consider more general kinds of \mbox{conuclear} and nuclear images. This problem has a multitude of variants and subcases: in some cases we provide satisfactory answers, in other cases we only offer open problems.

  In posing this question, we continue the line of research of Galatos and Tsinakis~\cite{galatos+tsinakis05}, who obtained the answer for conuclear images of a particular kind. While a general \emph{conucleus} on a partially ordered monoid with a multiplicative unit $\1$ is an interior operator $\sigma$ such that ${\sigma(a) \cdot \sigma(b) \leq \sigma(a \cdot b)}$ and $\sigma(\1) = \1$ and a \emph{nucleus} is a closure operator $\gamma$ such that $\gamma(a) \cdot \gamma(b) \leq \gamma(a \cdot b)$, a \emph{kernel} is a conucleus whose image is downward closed. Galatos and Tsinakis introduced the variety of GMV-algebras as a generalization of MV-algebras and proved that the kernel images of $\ell$-groups are precisely the cancellative GMV-algebras and the nuclear images of cancellative GMV-algebras are precisely GMV-algebras. These algebras need not be commutative, integral, or even bounded, but they still behave very much like MV-algebras in some ways.

  At the object level, the proof of Galatos and Tsinakis relies on a construction due to Bosbach~\cite{bosbach82}, which embeds every so-called cone algebra into the negative cone of an $\ell$-group. (Cone algebras form a variety of \mbox{algebras} whose signature consists of two binary division-like operations. This variety contains in particular the division reducts, i.e.\ implication reducts, of MV-algebras and pseudo MV-algebras.) To obtain a categorical equivalence, they then extend Mundici's notion of a good sequence to GMV-algebras. The proof of Dvure\v{c}enskij~\cite{dvurecenskij02} also relies ultimately on Bosbach's cone algebra construction, reformulated in terms of so-called semiclans~\cite{bosbach81}. The proof of Dvure\v{c}enskij and Vetterlein~\cite{dvurecenskij+vetterlein01b}, on the other hand, bypasses cone algebras entirely and instead uses the construction of the free semigroup over a partial semigroup originally due to Baer~\cite{baer49}.

  We move beyond the setting of GMV-algebras here. This requires us to abandon techniques relying on good sequences and cone algebras. The key tool will be the construction of the \emph{free nuclear preimage} of a partially ordered or semilattice-ordered monoid, which is left adjoint to the functor which sends a partially ordered or semilattice-ordered monoid equipped with a nucleus to its nuclear image. This construction involves taking arbitrary sequences and quotienting by a suitable preorder, rather than restricting to well-behaved sequences as in the good sequence construction. In this regard, our construction is similar in spirit to the approach of Dvure\v{c}enskij and Vetterlein (see the remarks on p.~\pageref{page: baer} at the end of Section~\ref{sec: free po preimages}). It needs to be emphasized that when applied to a finite MV-algebra, the construction introduced here does \emph{not} coincide with Mundici's construction. Instead, the negative cone of Mundici's $\ell$-group is a proper homomorphic image of the (non-divisible) integral cancellative commutative residuated lattice constructed here.

  For the reader's convenience, let us summarize the above discussion in diagrammatic form. The original theorem of Mundici can (anachronistically) be thought of as a way of inverting the following sequence of constructions:
\begin{align*}
  \text{Abelian $\ell$-group} ~ \xrightarrow{\text{negative cone}} ~ \text{cancellative ICGMV} ~ \xrightarrow{\text{unit interval}} ~  \text{MV,}
\end{align*}
  where cancellative ICGMV-algebras (integral commutative GMV-algebras) are precisely the negative cones of Abelian $\ell$-groups. Galatos and Tsinakis then refined Mundici's construction and used it to invert a more general sequence of constructions:
\begin{align*}
  \text{$\ell$-group} ~ \xrightarrow{\text{kernel}} ~ \text{cancellative GMV} ~ \xrightarrow{\text{nucleus}} ~  \text{GMV.}
\end{align*}
  What we are after in this paper is a way of inverting a still more general sequence:
\begin{align*}
  \text{$\ell$-group} ~ \xrightarrow{\text{conucleus}} ~ \text{cancellative RL} ~ \xrightarrow{\text{nucleus}} ~  \text{integrally closed RL,}
\end{align*}
  where the variety of integrally closed residuated lattices, studied in detail in~\cite{gil-ferez+lauridsen+metcalfe20}, is axiomatized by the equations $x \bs x \approx \1 \approx x / x$.\footnote{Interestingly, the authors of the paper~\cite{gil-ferez+lauridsen+metcalfe20} establish a relationship between integrally closed residuated lattices and $\ell$-groups going in the opposite direction to the one considered here: they show that the image of the double negation nucleus of an integrally closed residuated lattice is an $\ell$-group.} We do not fully achieve this goal, but we do obtain some partial answers and develop some useful tools.

\subsection*{Preliminaries}

  Before we sketch the main results of the paper, let us briefly lay out some basic definitions. A \emph{partially ordered monoid (pomonoid)} $\alg{M} = \langle M, \leq, \cdot, \1 \rangle$ is a monoid $\langle M, \cdot, \1 \rangle$ with a partial order $\langle M, \leq \rangle$ such that the monoidal multiplication is isotone in both arguments with respect to this order. A \emph{semilattice-ordered monoid (s$\ell$-monoid)} is both a join semi\-lattice and a pomonoid with respect to the semilattice order such~that
\begin{align*}
  a \cdot (b \vee c) & = (a \cdot b) \vee (a \cdot c), & (a \vee b) \cdot c & = (a \cdot c) \vee (b \cdot c).
\end{align*}
  A pomonoid or s$\ell$-monoid is called \emph{integral} if $\1$ is the top element of $\alg{M}$.

  A \emph{conucleus} on a pomonoid $\alg{M}$ is an interior operator $\sigma$ on $\alg{M}$, i.e.\ a monotone map $\sigma\colon \alg{M} \to \alg{M}$ with $\sigma(\sigma(a)) = \sigma(a) \leq a$, such that
\begin{align*}
  \sigma(a) \cdot \sigma(b) & \leq \sigma(a \cdot b).
\end{align*}
  The \emph{conuclear image} of $\alg{M}$ (also called the \emph{conucleus image}) with respect to $\sigma$ is the pomonoid $\alg{M}_{\sigma}$ consisting of the $\sigma$-open elements of $\alg{M}$ with the order and multi\-plication inherited from~$\alg{M}$. In other words, $\alg{M}_{\sigma}$ is the subpomonoid of elements $a$ such that $\sigma(a) = a$. A \emph{nucleus} on $\alg{M}$, on the other hand, is a closure operator $\gamma$ on $\alg{M}$, i.e.\ a monotone map $\gamma\colon \alg{M} \to \alg{M}$ with $a \leq \gamma(a) = \gamma(\gamma(a))$, such that
\begin{align*}
  \gamma(a) \cdot \gamma(b) & \leq \gamma(a \cdot b).
\end{align*}
  The \emph{nuclear image} of $\alg{M}$ (also called the \emph{nucleus image}) with respect to $\gamma$ is the pomonoid $\alg{M}_{\gamma}$ consisting of the $\gamma$-closed elements of $\alg{M}$ with the order inherited from~$\alg{M}$, the multiplication
\begin{align*}
  a \cdot_{\gamma} b & \assign \gamma(a \cdot b),
\end{align*}
  and the multiplicative unit $\gamma(\1)$.

  If $\alg{M}$ is an s$\ell$-monoid, then so are $\alg{M}_{\sigma}$ and $\alg{M}_{\gamma}$. Joins in $\alg{M}_{\sigma}$ coincide with joins in~$\alg{M}$, while joins in $\alg{M}_{\gamma}$ are computed as $a \vee_{\gamma} b \assign \gamma(a \vee b)$. If $\alg{M}$ is a meet semilattice, then so are $\alg{M}_{\sigma}$ and $\alg{M}_{\gamma}$. Meets in $\alg{M}_{\sigma}$ are computed as $a \wedge_{\sigma} b \assign \sigma(a \wedge b)$, while meets in $\alg{M}_{\gamma}$ coincide with meets in $\alg{M}$.

  Crucially, both of these constructions preserve the existence of residuals of multiplication. A \emph{residuated po\-monoid} $\alg{M} = \langle M, \leq, \1, \cdot, \bs, / \rangle$ is a pomonoid $\langle M, \leq, \cdot, \1 \rangle$ with two additional binary operations $x \bs y$ and $y / x$, called the \emph{residuals} of multiplication, such that
\begin{align*}
  b \leq a \bs c \iff a \cdot b \leq c \iff a \leq c / b.
\end{align*}
  A \emph{residuated lattice} is then an algebra which is both a lattice and a residuated pomonoid with respect to the lattice order. If multiplication in $\alg{M}$ has residuals, then so does multiplication in $\alg{M}_{\sigma}$, although the residuals are in general different: $a \bs_{\sigma} b \assign \sigma(a \bs b)$ and $b /_{\!\sigma} a \assign \sigma(b / a)$. Moreover, $\alg{M}_{\gamma}$ is also a residuated pomonoid and the residuals of $\alg{M}_{\gamma}$ coincide with the residuals of~$\alg{M}$. In particular, if $\alg{M}$ is a residuated lattice, then so are $\alg{M}_{\sigma}$ and $\alg{M}_{\gamma}$, with
\begin{align*}
  & \alg{M}_{\sigma} \assign \langle M_{\sigma}, \leq, \cdot, \1, \wedge_{\sigma}, \vee, \bs_{\sigma}, /_{\!\sigma} \rangle, & & \alg{M}_{\gamma} \assign \langle M_{\gamma}, \leq, \cdot_{\gamma}, \gamma(\1), \wedge, \vee_{\gamma}, \bs, / \rangle.
\end{align*}

  Nuclei and conuclei can also defined for posemigroups. However, to avoid cluttering the paper with duplicate definitions and results, for the most part we restrict our explicit attention to monoidal structures. It is generally a simple matter to obtain analogous definitions and results for structures based on semigroups by omitting all reference to the multiplicative unit. We only discuss structures based on posemigroups if there is a divergence with structures based on pomonoids that is worth remarking on.

\subsection*{Main results}

  With the requisite definitions out of the way, we may now succinctly state our main question:
\begin{align*}
   \text{Which residuated lattices arise as $(\alg{G}_{\sigma})_{\gamma}$ for some (Abelian) $\ell$-group $\alg{G}$?}
\end{align*}
  This is really an entire family of questions, depending on the constraints we place on $\sigma$, $\gamma$, and $\alg{G}$. The results of Mundici, Dvure\v{c}enskij, and Vetterlein on MV-algebras and pseudo MV-algebras cover the case where $\sigma$ is the negative cone co\-nucleus $\sigma_{-}(a) \assign \1 \wedge a$ and $\gamma$ is a unit interval nucleus $\gamma_{u} \assign u \vee a$, while Galatos and Tsinakis cover the case where the image of $\sigma$ is a downward closed subset of $\alg{G}$.

  In these cases, not only can residuated lattices in the appropriate varieties be represented in terms of $\ell$-groups, but in fact a categorical equivalence connects these varieties of GMV-algebras to classes of $\ell$-groups with additional structure. We do not aim for such categorical equivalences here: while the construction studied in the present paper allows us to go significantly beyond the cases covered by the constructions of Mundici and others, it provides a substantially weaker link between the two sides.

  The above question naturally splits into two parts. The first part of the problem has been studied by \mbox{Montagna} and Tsinakis~\cite{montagna+tsinakis10}, who identified the (integral) conuclear images of Abelian $\ell$-groups as precisely the (integral) commutative cancellative residuated lattices. Here a residuated lattice, and more generally a pomonoid, is called \emph{cancellative} if it satisfies the two implications
\begin{align*}
  a \cdot x \leq b \cdot x & \implies a \leq b, &
  x \cdot a \leq x \cdot b & \implies a \leq b.
\end{align*}
  Beyond the commutative case, however, describing the conuclear images of $\ell$-groups is much more difficult. (Compare the straightfoward fact that each commutative cancellative monoid embeds into an Abelian group with the complicated description of cancellative monoids which embed into a group~\cite{clifford+preston67}.) We therefore opt to answer a different but related question, namely:
\begin{align*}
   \text{Which residuated lattices arise as $\alg{L}_{\gamma}$ for some cancellative residuated lattice $\alg{L}$?}
\end{align*}

  We shall in fact treat this as a problem concerning pomonoids and s$\ell$-monoids. The bulk of our work will focus on determining which pomonoids and s$\ell$-monoids are the nuclear images of cancellative pomonoids and s$\ell$-monoids. Results about residuated lattices will be derived as corollaries, once we observe that the cancellative s$\ell$-monoids obtained in this way from \emph{finite} residuated lattices are residuated lattices themselves.

  At the most basic level, our answers to the above questions thus rely on first answering the following question for partially ordered groups (pogroups):
\begin{align*}
  \text{Which pomonoids arise as $\alg{M}_{\gamma}$ for some subpomonoid $\alg{M}$ of a pogroup $\alg{G}$?}
\end{align*}
  These turn out to be precisely the \emph{integrally closed} pomonoids, i.e.\ pomonoids which satisfy the following two implications:
\begin{align*}
  a \cdot x \leq a & \implies x \leq \1, &
  x \cdot a \leq a & \implies x \leq \1.
\end{align*}
  Observe that in particular every integral pomonoid is integrally closed, and so is every cancellative pomonoid.

\begin{theoremA*}[cf.\ Theorem~\ref{thm: integrally closed pomonoids}]
  The nuclear images of subpomonoids of (Abelian) pogroups are precisely the (commutative) integrally closed pomonoids.
\end{theoremA*}

\begin{theoremB*}[cf.\ Theorem~\ref{thm: integrally closed pomonoids}]
  The nuclear images of subpomonoids of negative cones of (Abelian) pogroups are precisely the (commutative) integral pomonoids.
\end{theoremB*}

  The same questions arise for s$\ell$-monoids and $\ell$-groups. In the Abelian case, we describe the appropriate class of nuclear images by an infinite set of s$\ell$-monoidal equations, which we call the \emph{square condition} (see Definition~\ref{def: square condition}). Beyond the Abelian case, the gap which arises between cancellative structures and submonoids of groups throws a wrench into our argument, forcing us to make do with cancellative s$\ell$-monoids rather than sub-s$\ell$-monoids of $\ell$-groups.

\begin{theoremC*}[cf.\ Theorem~\ref{thm: commutative integrally closed square sl-monoids}]
  The nuclear images of (integral) sub-s$\ell$-monoids of Abelian $\ell$-groups are precisely the commutative integrally closed (integral) s$\ell$-monoids which satisfy the square condition.
\end{theoremC*}  

\begin{theoremD*}[cf.\ Theorem~\ref{thm: integral sl-monoids}]
  The nuclear images of integral cancellative s$\ell$-monoids are precisely the integral s$\ell$-monoids.
\end{theoremD*}

  In the case of pomonoids, we closed the above-mentioned gap by means of a proof-theoretic argument. One may hope that with additional work, a similar proof-theoretic argument can be found at the level s$\ell$-monoids. This is one major task left open by the present paper. The other principal problem left open is to describe the nuclear images of cancellative s$\ell$-monoids (beyond the integral case and the commutative case). Again, one may then hope to use a proof-theoretic argument to extend this to a description of the nuclear images of sub-s$\ell$-monoids of $\ell$-groups. If these tasks can be accomplished, and if the nuclear images of cancellative s$\ell$-monoids turn out to be the integrally closed s$\ell$-monoids, then the payoff would be a proof of the following conjecture.

\begin{conjectureA*}
  The nuclear images of sub-s$\ell$-monoids of (negative cones of) $\ell$-groups are precisely the integrally closed (integral) s$\ell$-monoids.
\end{conjectureA*}

  Results about residuated lattices now follow from the above results about s$\ell$-monoids, provided that we restrict to nuclear images which are finite (or more generally, dually well-partially-ordered). Each finite (and more generally, each dually well-partially-ordered) integrally closed residuated lattice is in fact integral, which is why the integrally closed condition does not explicitly occur in the following theorems.

\begin{theoremE*}[cf.\ Theorem~\ref{thm: finite integral commutative square rls}]
  The \emph{finite} nuclear images of commutative conuclear \mbox{images} of (negative cones of) Abelian $\ell$-groups are precisely the finite commutative integral residuated lattices which satisfy the square condition.
\end{theoremE*}

\begin{theoremF*}[cf.\ Theorem~\ref{thm: finite integral cancellative rls}]
  The \emph{finite} nuclear images of (integral) cancellative residuated lattices are precisely the finite integral residuated \mbox{lattices}.
\end{theoremF*}

  One might again wish to replace cancellative residuated lattices by conuclear images of $\ell$-groups in this theorem. We do not know whether this is possible, but let us record this as a conjecture.

\begin{conjectureB*}
  The \emph{finite} nuclear images of conuclear images of (negative cones of) $\ell$-groups are precisely the finite integral residuated lattices.
\end{conjectureB*}

  We in fact exhibit each finite integral residuated lattice as a nuclear image of an integral cancellative residuated lattice which is moreover distributive and satisfies
\begin{align*}
  & x (y \wedge z) = xy \wedge xz, & & (x \wedge y) z = xz \wedge yz.
\end{align*}
  Montagna and Tsinakis~\cite{montagna+tsinakis10} showed that integral cancellative \emph{commutative} residuated lattices satisfying distributivity and the above equations are precisely the conuclear images of Abelian $\ell$-groups with respect to a conucleus $\sigma$ such that $\sigma(x \wedge y) = \sigma(x) \wedge \sigma(y)$. This leads us to suspect that an even stronger conjecture than the last one stated above might hold, namely that each finite integral residuated lattice may be a nuclear image of an image of an $\ell$-group with respect to a conucleus $\sigma$ such that $\sigma(x \wedge y) = \sigma(x) \wedge \sigma(y)$.

\subsection*{Free nuclear preimages}

  How do we prove the above theorems? The key con\-struction is the \emph{free nuclear preimage} of a pomonoid. This functorial construction is left adjoint to the nuclear image functor, which takes a pomonoid equipped with a nucleus $\pair{\alg{M}}{\gamma}$ to the nuclear image~$\alg{M}_{\gamma}$.

  We start by taking the free pomonoid $\alg{F}({\alg{M}}) = \langle \FreeMon{M}, \leq, \circ, \emptyword \rangle$ over the underlying poset of $\alg{M}$. That is, $\alg{F}(\alg{M})$ is the monoid of words over $\alg{M}$ ordered as follows:
\begin{align*}
  [a_{1}, \dots, a_{m}] \leq [b_{1}, \dots, a_{n}] & \iff m = n \text{ and } a_{i} \leq b_{i} \text{ for } 1 \leq i \leq m = n.
\end{align*}
 There is a canonical multiplication homomorphism $\gamma\colon \alg{F}(\alg{M}) \to \alg{M}$:
\begin{align*}
  \gamma ([a_{1}, \dots, a_{n}]) & \assign a_{1} \cdot \ldots \cdot a_{n}, & \gamma(\emptyword) & \assign \1.
\end{align*}
  This determines an isotone map $[\gamma]\colon \alg{F}(\alg{M}) \to \alg{F}(\alg{M})$ which reduces words of arbitrary length to words of length one, namely $[\gamma](w) \assign [\gamma(w)]$.

  There is a smallest order congruence $\sqleq$ on $\alg{F}(\alg{M})$, i.e.\ a preorder compatible with multiplication which extends the partial order of $\alg{F}(\alg{M})$, such that $[\gamma]$ is a nucleus with respect to $\sqleq$, i.e.\
\begin{align*}
  u \sqleq [\gamma](u) \sqleq [\gamma]([\gamma](u)), & &[\gamma](u) \circ [\gamma](v) \sqleq [\gamma](u \circ v).
\end{align*}
  The free nuclear preimage of $\alg{M}$, denoted $\FreeMon{\alg{M}}$, is obtained from the preordered structure $\langle \FreeMon{M}, \sqleq, \circ, \emptyword, [\gamma]$ by taking its quotient with respect to the equivalence relation associated with the preorder.

  The \emph{free semilattice-ordered nuclear preimage} of an s$\ell$-monoid turns out to be the s$\ell$-monoid $\Id \FreeMon{\alg{M}}$ of non-empty finitely generated down\-sets of the free partially ordered nuclear preimage $\FreeMon{\alg{M}}$. The results about finite residuated lattices stated above then follow from results about s$\ell$-monoids once we observe that the free semilattice-ordered nuclear preimage of a \emph{finite} (or more generally, dually well-partially-ordered) residuated lattice is a residuated lattice.

  The free partially ordered nuclear preimage $\FreeMon{\alg{M}}$ is cancellative if and only if $\alg{M}$ is integrally closed. In fact, in that case $\FreeMon{\alg{M}}$ embeds into a partially ordered group. The cancellativity of the free semilattice-ordered nuclear preimage $\Id \FreeMon{\alg{M}}$ is a more delicate question, as we shall see. Nevertheless, even if $\Id \FreeMon{\alg{M}}$ itself is not cancellative, we may use it to show that every integral s$\ell$-monoid $\alg{M}$ is the nuclear image of some cancellative s$\ell$-monoid (not necessarily of $\Id \FreeMon{\alg{M}}$).

  A welcome by-product of the above line of investigation is that the free \mbox{nuclear} preimage construction \mbox{allows} us to provide a syntactic description of ordered quasivarieties of pomonoids and s$\ell$-monoids closed under nuclear images. A~\emph{po\-monoidal (s$\ell$-monoidal) quasi-inequation} is an implication of the form
\begin{align*}
  t_{1} \leq u_{1} ~ \& ~ \dots ~ \& ~ t_{n} \leq u_{n} \implies t \leq u,
\end{align*}
  where $t_{i}$, $u_{i}$, $t$, $u$ are terms in the signature of monoids (s$\ell$-monoids). An \emph{ordered quasivariety} of pomonoids (s$\ell$-monoids) is then a class axiomatized by a set of quasi-inequations. Of course, ordered quasi\-varieties of s$\ell$-monoids are simply quasi\-varieties in the ordinary sense of the word.

  We call a quasi-inequation \emph{simple} if it each $u_{i}$ is a variable. That is, products, units, and joins are not allowed to occur in the right-hand sides of the premises.

\begin{theoremG*}[cf.\ Theorems~\ref{thm: nuclear images of pomonoids} and \ref{thm: nuclear images of sl-monoids}]
  An ordered quasivariety of pomonoids (a quasivariety of s$\ell$-monoids) is closed under nuclear images if and only if it is axiomatized by simple quasi-inequations.
\end{theoremG*}

  This theorem ultimately reflects the following property of free partially ordered nuclear preimages:
\begin{align*}
  u \sqleq v \circ w \implies u_{1} \sqleq v \text{ and } u_{2} \sqleq w \text{ for some decomposition } u_{1} \circ u_{2} = u,
\end{align*}
  and the following property of free semilattice-ordered nuclear preimages:
\begin{align*}
  u \sqleq v \vee w \implies u_{1} \sqleq v \text{ and } u_{2} \sqleq w \text{ for some decomposition } u_{1} \vee u_{2} = u.
\end{align*}

  The reader will have observed that simple quasi-inequations can be defined in any signature. The problem of describing the ordered subquasivarieties which are axiomatizable by simple quasi-inequations therefore arises for any ordered quasivariety. However, we shall not pursue this general problem in the present paper, given that our main goal here is to relate pomonoids and pogroups.

  Observe also that the analogues of the above theorems and constructions for conuclear images instead of nuclear images are trivial: the free conuclear image of a pomonoid (s$\ell$-monoid) is the pomonoid (s$\ell$-monoid) itself with the identity conucleus, and every quasi-inequation in the signature of pomonoids (s$\ell$-monoids) is preserved under conuclear images, since it is preserved under subpomonoids (sub-s$\ell$-monoids).

\section{Free nuclear preimages of pomonoids}
\label{sec: free po preimages}

  In this section we study \emph{free nuclear preimages} of pomonoids. As described above, the \mbox{nuclear} image construction assigns to each a pair consisting of a pomonoid $\alg{M}$ and a nucleus $\gamma$ a pomonoid $\alg{M}_{\gamma}$. Such pairs $\pair{\alg{M}}{\gamma}$ will be called \emph{nuclear pomonoids}. A homomorphism of nuclear pomonoids $h\colon \langle \alg{M}, \gamma \rangle \to \langle \alg{N}, \delta \rangle$ is then a homomorphism of pomonoids ${h\colon \alg{M} \to \alg{N}}$ such that $h(\gamma(x)) = \delta(h(x))$. Such a homomorphism then restricts to a homomorphism of pomonoids $h {\mid_{\scriptscriptstyle\alg{M}_{\gamma}}}\colon \alg{M}_{\gamma} \to \alg{N}_{\delta}$. This yields the \emph{nuclear image functor} from the category $\PomonStar$ of nuclear pomonoids to the category $\Pomon$ of pomonoids. The free nuclear preimage functor is the left adjoint of this functor.

  Nuclei satisfying the additional condition that $\gamma(\1) = \1$ will be of particular interest to us. Such a nucleus will be called \emph{unital} and the nuclear pomonoid $\pair{\alg{M}}{\gamma}$ will be called a \emph{unital} nuclear pomonoid. Restricting to such nuclear pomonoids yields the \emph{unital nuclear image functor}. We also describe the left adjoint of this functor, which we call the \emph{free unital nuclear preimage functor}.

  Starting with a pomonoid $\alg{M}$, we consider the free monoid (the monoid of words) over M and endow it with a pre\-order~$\sqleq$ and a map~$[\gamma]$. This yields a preordered structure~$\FreeMon{\alg{M}} \assign {\langle \FreeMon{M}, \sqleq, \circ, \emptyword, [\gamma] \rangle}$. The~free nuclear pre\-image of $\alg{M}$ is then obtained by collapsing this preordered structure to partially ordered one, which we also denote by $\FreeMon{\alg{M}}$. If we restrict to non-empty words and shift the multi\-plicative unit from the empty word $\emptyword$ to the singleton word $[\1]$, we obtain the free unital nuclear preimage $\FreeUMon{\alg{M}}$, which behaves better with respect to residuation and cancellation than $\FreeMon{\alg{M}}$.

  If the pomonoid $\alg{M}$ is commutative, we can replace the free monoid by the free commutative monoid in the above constructions, resulting in the free commutative nuclear preimage $\CFreeRedMon{\alg{M}}$ and the free commutative unital nuclear preimage $\CFreeRedUMon{\alg{M}}$.

  Throughout the following, let $\alg{M} = \langle M, \leq, \cdot, \1 \rangle$ be a pomonoid and let $\FreeMon{M}$ be the set of words over the alphabet $M$. The empty word will be denoted by~$\emptyword$ and the word consisting of the letters $a_{1}, \dots, a_{n} \in \alg{M}$ in this order will be denoted by $[a_{1}, \dots, a_{n}]$. Thus $[a_{1}, \dots, a_{m}] \circ [b_{1}, \dots, b_{n}] = [a_{1}, \dots, a_{m}, b_{1}, \dots, b_{n}]$.

  We now repeat some definitions which were already stated in the introduction. We can partially order the set of words $\FreeMon{M}$ as follows:
\begin{align*}
  [a_{1}, \dots, a_{m}] \leq [b_{1}, \dots, a_{n}] & \iff m = n \text{ and } a_{i} \leq b_{i} \text{ for } 1 \leq i \leq m = n.
\end{align*}
  This yields the pomonoid $\alg{F}(\alg{M}) \assign \langle \FreeMon{M}\!, \leq, \circ, \emptyword \rangle$. This is in fact the free pomonoid generated by the poset reduct of~$\alg{M}$. It is equipped with an \emph{evaluation map}~$\gamma$ which sends each word to the corresponding product in $\alg{M}$. In other words, we have a homomorphism of pomonoids $\gamma\colon \alg{F}(\alg{M}) \to \alg{M}$ defined as:
\begin{align*}
  \gamma (\emptyword) & = \1, \\
  \gamma ([a_{1}, \dots, a_{n}]) & = a_{1} \cdot \ldots \cdot a_{n}.
\end{align*}
  The map $[\gamma]\colon \alg{F}(\alg{M}) \to \alg{F}(\alg{M})$ is then defined as $[\gamma](u) \assign [\gamma(u)]$. The range of the map $[\gamma]$ is therefore the set of all singleton words (see Figure~\ref{fig: nuclear preimage}).

  We now define a preorder $\sqleq$ on $\alg{F}(\alg{M})$ with respect to which $[\gamma]$ is a nucleus:
\begin{align*}
  u \sqleq \emptyword \iff & u = \emptyword, \\
  u \sqleq [a_{1}, \dots, a_{n}] \iff & \text{there are } u_{1}, \dots, u_{n} \in \alg{F}(\alg{M}) \text{ such that } \\
  & u = u_{1} \circ \ldots \circ u_{n} \text{ and } \gamma(u_{i}) \leq a_{i} \text{ for each } u_{i}.
\end{align*}
  Note that the words $u_{i}$ in the decomposition of $u$ may be empty. Equivalently, this is the smallest preorder compatible with the multiplication of $\alg{F}(\alg{M})$ which extends $\leq$ and satisfies also $u \sqleq [\gamma](u)$ for each word $u$.

  The following observations are now immediate.

\begin{figure}
\caption{The free nuclear preimage of $\alg{M}$}
\label{fig: nuclear preimage}
\smallskip
\begin{center}
\begin{tikzpicture}[scale=1.25, dot/.style={circle,fill,inner sep=1.5pt,outer sep=1.5pt}]

  \draw [line width=0.25mm,thick] {(0,0.4) ellipse (2.5 and 2.75)};
  \draw [line width=0.25mm,name path=classellipse,dashed,thick] {(0,1) ellipse (1 and 1.45)};
  \node at (-1,2.2) {$[\alg{M}]$};
  \node at (-1.6,2.9) {$\FreeMon{\alg{M}}$};
  \node at (-0.5,1.75) {$[a]$};
  \node at (0.5,1.75) {$[b]$};
  \node at (0,0.5) {$[a \cdot b]$};
  \node at (0,-1.5) {$[a, b] = [a] \circ [b]$};
  \node (a) at (-0.5,1.45) [dot] {};
  \node (b) at (0.5,1.45) [dot] {};
  \node (adotb) at (0,0.2) [dot] {};
  \node (acircb) at (0,-1.2) [dot] {};
  \draw[->] (acircb.north east) to [out=45,in=-45] (adotb.east);
  \node (gamma) at (0.6,-0.6) {$[\gamma]$};

\end{tikzpicture}
\end{center}
\end{figure}
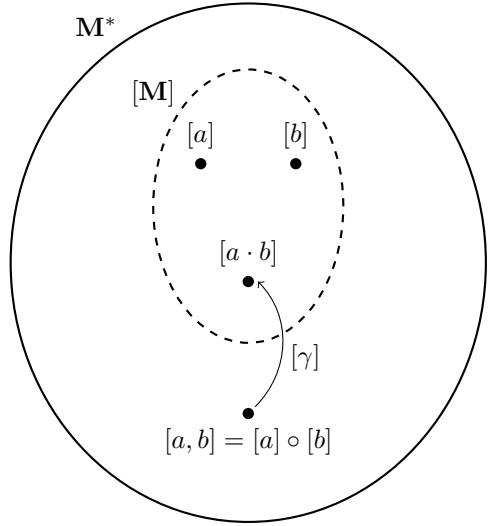

\begin{fact}
  $u \sqleq [a]$ if and only if $\gamma(u) \leq a$ in $\alg{M}$.
\end{fact}

\begin{fact}
  $\emptyword \sqleq u \circ v$ if and only if $\emptyword \sqleq u \text{ and } \emptyword \sqleq v$.
\end{fact}

\begin{fact}
  The relation $\sqleq$ is a preorder. Multiplication is isotone with respect to this preorder: if $u \sqleq u'$ and $v \sqleq v'$, then $u \circ v \sqleq u' \circ v'$. The map $[\gamma]$ is also isotone with respect to this preorder: if $u \sqleq v$, then $[\gamma(u)] \sqleq [\gamma(v)]$.
\end{fact}

\begin{proof}
  The relation $\sqleq$ is reflexive: by definition $\emptyword \sqleq \emptyword$, and the decomposition of $u = [a_{1}, \dots, a_{n}]$ as $[a_{1}] \circ \ldots \circ [a_{n}]$ witnesses that $u \sqleq u$. The relation $\sqleq$ is transitive: suppose that $u \sqleq v \sqleq w$. If $w = \emptyword$, then $v = \emptyword$. If $v = \emptyword$, then $u = \emptyword = v$, hence $u \sqleq w$. We may therefore assume that $v = [b_{1}, \dots, b_{m}]$ and $w = [c_{1}, \dots, c_{n}]$. Let $u = u_{1} \circ \ldots \circ u_{m}$ and $v = v_{1} \circ \ldots \circ v_{n}$ be decompositions of $u$ and $v$ witnessing the inequalities $u \sqleq v$ and $v \sqleq w$. To each $v_{i} = [b_{p}, \dots, b_{q}]$ we can assign the word $u'_{i} = u_{p} \circ \ldots \circ u_{q}$ so that $u = u'_{1} \circ \ldots \circ u'_{n}$, taking $u'_{i} = \emptyword$ if $v_{i} = \emptyword$. Then $\gamma(u'_{i}) = \gamma(u_{p}) \cdot \ldots \cdot \gamma(u_{q}) \leq b_{p} \cdot \ldots \cdot b_{q} = \gamma (v_{i}) \leq c_{i}$ if $u'_{i}$ is non-empty and $\gamma(u'_{i}) = \1 = \gamma(v_{i}) \leq c_{i}$ if $u'_{i} = \emptyword$, hence $u \sqleq w$. Multiplication is isotone: combining the decompositions which witness that $u \sqleq u'$ and $v \sqleq v'$ yields a decomposition witnessing that $u \circ v \sqleq u' \circ v'$. Finally, the map $[\gamma]$ is isotone: if~$u \sqleq v = \emptyword$, then $u = \emptyword$ and $[\gamma(u)] = [\1] = [\gamma(v)]$, and if the decomposition $u = u_{1} \circ \ldots \circ u_{n}$ witnesses that $u \sqleq v = [b_{1}, \dots, b_{n}]$, then $[\gamma (u)] = [\gamma(u_{1}) \cdot \ldots \cdot \gamma(u_{n})] \sqleq \linebreak[1] [b_{1} \cdot \ldots \cdot b_{n}] = [\gamma (v)]$.
\end{proof}

  This preorder is never a partial order: $[a] \sqleq [\1, a] \sqleq [a]$ even though $[a]$ and $[\1, a]$ are distinct words for each $a \in \alg{M}$. To obtain a pomonoid from the preordered structure $\langle M^{*}, \sqleq, \circ, \emptyword \rangle$, we need to take the quotient by the equivalence relation induced by the preorder, i.e.\ by the relation $u \sim v \iff u \sqleq v \text{ and } v \sqleq u$. This yields a pomonoid denoted $\FreeMon{\alg{M}}$.

  The elements of $\FreeMon{\alg{M}}$ are thus equivalence classes of words over $M$. However, we shall talk about elements of $\FreeMon{\alg{M}}$ as if they were words over~$\alg{M}$, with the implicit understanding that we are in fact referring to the corresponding equivalence classes. For example, we shall say that the unit of $\FreeMon{\alg{M}}$ is $\emptyword$, meaning that it is the equivalence class $[\emptyword]_{\sim} = \{ \emptyword \}$ of $\emptyword$ with respect to the above equivalence relation. This convention will greatly simplify our proofs by allowing us to work directly with the preordered structure $\langle M^{*}, \sqleq, \circ, \emptyword \rangle$.

\begin{fact}
  The map $[\gamma]$ is a nucleus on the pomonoid $\FreeMon{\alg{M}}$.
\end{fact}

\begin{proof}
  The map $[\gamma]$ is increasing: the inequality $u \sqleq [\gamma (u)]$ is witnessed by the de\-com\-position of $u$ into a single block $u = u_1$. The map $[\gamma]$ is idempotent: $[\gamma([\gamma(u)])] = [\gamma(u)]$ because $[\gamma([a])] = [a]$ for $a \in \alg{M}$. Finally, if $u = [a_{1}, \dots, a_{m}]$ and $v = [b_{1}, \dots, b_{n}]$, then $[\gamma (u)] \circ [\gamma (v)] = [a_{1} \cdot \ldots \cdot a_{m}, b_{1} \cdot \ldots \cdot b_{n}] \sqleq [a_{1} \cdot \ldots \cdot a_{m} \cdot b_{1} \cdot \ldots \cdot b_{n}] = [\gamma (u \circ v)]$. If $u = \emptyword$, then $[\gamma(u)] \circ [\gamma (v)] = [\gamma(\emptyword)] \circ [\gamma (v)] = [\1] \circ [\gamma(v)] \sqleq [\gamma(v)] = [\gamma(\emptyword \circ v)] = [\gamma(u \circ v)]$, and likewise for $v = \emptyword$.
\end{proof}

\begin{fact}
  The nuclear image of $[\gamma]$ on $\FreeMon{\alg{M}}$ is isomorphic to $\alg{M}$ via $a \mapsto [a]$.
\end{fact}

\begin{proof}
  The range of $[\gamma]$ consists of the elements of the form $[a]$ for some $a \in \alg{M}$, and $[a] \sqleq [b] \iff a \leq b$, therefore the map is an order isomorphism. Moreover, $[a] \circ_{[\gamma]} [b] = [\gamma ([a] \circ [b])] = [a \cdot b]$ and $[\gamma(\emptyword)] = [\1]$.
\end{proof}

  Observe that $\emptyword$ and $[\1]$ behave almost identically in $\FreeMon{\alg{M}}$:
\begin{align*}
  \emptyword \circ w = [\1] \circ w = w = w \circ [\1] = w \circ \emptyword
\end{align*}
  whenever $w$ is a non-empty word. Structure $\FreeUMon{\alg{M}}$ obtained from~$\FreeMon{\alg{M}}$ by restricting to (the equivalence classes of) non-empty words is thus a pomonoid whose unit is $[\1]$ rather than $\emptyword$. In fact, it is a unital nuclear pomonoid, since $[\gamma]([\1]) = [\1]$.

  If $\alg{M}$ is a commutative pomonoid, we can replace the free monoid over $M$ by the free commutative monoid in the above construction, resulting in the commutative pomonoid $\CFreeMon{\alg{M}}$. We can still use the notation $[a_{1}, \dots, a_{n}]$ for elements of $\CFreeMon{\alg{M}}$ if we keep in mind $[a_{1}, \dots, a_{n}] = [b_{1}, \dots, b_{n}]$ in $\CFreeMon{\alg{M}}$ if and only if $[b_{1}, \dots, b_{n}]$ is a permutation of $[a_{1}, \dots, a_{n}]$. The above proofs then go through word for word and we can again collapse $\CFreeMon{\alg{M}}$ to a commutative nuclear pomonoid $\CFreeRedMon{\alg{M}}$. Restricting to non-empty words yields the preordered structure $\CFreeUMon{\alg{M}}$ and the unital commutative nuclear pomonoid $\CFreeRedUMon{\alg{M}}$.
 
\begin{fact}
  Let $\alg{M}$ be a (commutative) integral pomonoid. Then the image of the nucleus $[\gamma]$ is an upset of $\FreeMon{\alg{M}}$ and $\FreeUMon{\alg{M}}$ (of $\CFreeRedMon{\alg{M}}$ and $\CFreeRedUMon{\alg{M}}$).
\end{fact}

\begin{proof}
  The image of $[\gamma]$ is the set of all words of length one. If~$[a] \sqleq w$, then there is a decomposition $w = w_{1} \circ [b] \circ w_{2}$ such that $\emptyword \sqleq w_{1}$ and $a \leq b$ and $\emptyword \sqleq w_{2}$. By the integrality of $\alg{M}$, $\emptyword \sqleq w_{1}$ implies $w_{1} = \emptyword$ or $w_{1} = [\1]$, and likewise for $w_{2}$. In either case $w$ is equivalent to $[b]$.
\end{proof}

\begin{fact}
  Each equivalence class in $\FreeMon{\alg{M}}$ has a unique shortest representative. 
\end{fact}

\begin{proof}
  Suppose that $u = [a_{1}, \dots, a_{n}]$ and $v = [b_{1}, \dots, b_{n}]$ are two shortest representatives of some equivalence class in $\FreeMon{\alg{M}}$. Then $u \sqleq v \sqleq u$. If the empty word occurs in the decomposition which witnesses one of these inequalities, say $u = u_{1} \circ \emptyword \circ u_{2}$ where $u_{1} \sqleq v_{1}$, $\emptyword \sqleq v_{2}$, $u_{2} \sqleq v_{3}$ and $v = v_{1} \circ v_{2} \circ v_{3}$ for non-empty $v_{2}$, then $u = u_{1} \circ u_{2} \sqleq v_{1} v_{3} \sqleq v_{1} v_{2} v_{3} = v \sqleq u$, hence $v_{1} v_{3}$ is a strictly shorter word in the same equivalence class. It follows that $\emptyword$ does not occur in the decompositions witnessing $u \sqleq v \sqleq u$, hence $a_{i} \leq b_{i} \leq a_{i}$ for $1 \leq i \leq n$ and $u = v$.
\end{proof}

  In case $\alg{M}$ is integral, this shortest representative is obtained by removing all instances of the letter $\1$ from a word, unless the word happens to be $[\1]$. We call such words \emph{$\1$-reduced}. That is, a word $w$ is $\1$-reduced if either $w = [\1]$ or $w$ does not contain $[\1]$ as a subword. If $u$ is non-empty and $v$ is a $\1$-reduced word, then each decomposition of $u$ witnessing that $u \sqleq v$ consists of \emph{non-empty} words. The following facts are immediate consequences of this observation.

\begin{fact}
  Let $\alg{M}$ be an integral pomonoid. If $u = [a_{1}, \dots, a_{m}] \sqleq [b_{1}, \dots, b_{n}] = v$ in $\FreeMon{\alg{M}}$, where $u$, $v$ are $\1$-reduced, then $m \geq n$. 
\end{fact}

\begin{fact} \label{fact: order on integral freeredmon}
  Let $\alg{M}$ be an integral pomonoid. Then $[a_{1}, \dots, a_{n}] \sqleq [b_{1}, \dots, b_{n}] = v$ in $\FreeMon{\alg{M}}$, where $u$, $v$ are $\1$-reduced, if and only if $a_{i} \leq b_{i}$ in $\alg{M}$ for all ${1 \leq i \leq n}$.
\end{fact}

  Similar remarks apply to integral commutative pomonoids.

\begin{fact}
  Let $\alg{M}$ be an integral commutative pomonoid. If $u = [a_{1}, \dots, a_{m}] \sqleq [b_{1}, \dots, b_{n}] = v$ in $\CFreeRedMon{\alg{M}}$ and $u$, $v$ are $\1$-reduced, then $m \geq n$. 
\end{fact}

\begin{fact} \label{fact: order on integral cfreeredmon}
  Let $\alg{M}$ be an integral commutative pomonoid. Then $[a_{1}, \dots, a_{n}] \sqleq [b_{1}, \dots, b_{n}] = v$ in $\CFreeRedMon{\alg{M}}$, where $u$, $v$ are $\1$-reduced, if and only if $a_{i} \leq c_{i}$ in $\alg{M}$ for all $1 \leq i \leq n$ for some permutation $[c_{1}, \dots, c_{n}]$ of $[b_{1}, \dots, b_{n}]$.
\end{fact}

  If $\alg{M}$ is integral, then $\FreeMon{\alg{M}}$ ($\FreeUMon{\alg{M}}$) is thus isomorphic to the pomonoid of (non-empty) $\1$-reduced words over $M$ where multiplication coincides with multi\-plication in $\alg{F}(\alg{M})$ except that $[\1] \circ u = u = u \circ [\1]$ if $u$ is non-empty. The same can be said about $\CFreeRedMon{\alg{M}}$ ($\CFreeRedUMon{\alg{M}}$) if the pomonoid $\alg{M}$ is moreover commutative.

\begin{fact}
  Let $\alg{M}$ be a (commutative) integral pomonoid. Then $\FreeUMon{\alg{M}}$ ($\CFreeUMon{\alg{M}}$) satisfies only equations which hold in all (commutative) monoids.
\end{fact}

\begin{proof} \label{fact: no equations preserved}
  The unordered reduct of $\FreeUMon{\alg{M}}$ ($\CFreeUMon{\alg{M}}$) is simply the free (commutative) monoid over the underlying set of $\alg{M}$ by Facts~\ref{fact: order on integral freeredmon} and \ref{fact: order on integral cfreeredmon}.
\end{proof}

  Let us now sum up all of the above constructions.

\begin{theorem}
  Let $\alg{M}$ be a pomonoid. Then
\begin{enumerate}[(i)]
\item $\FreeMon{\alg{M}} \assign \langle \FreeMon{M}, \sqleq, \circ, \emptyword, [\gamma] \rangle$ is a nuclear pomonoid,
\item $\FreeUMon{\alg{M}} \assign \langle \FreeUMon{M}, \sqleq, \circ, [\1], [\gamma] \rangle$ is a unital nuclear pomonoid,
\item $\CFreeRedMon{\alg{M}}$ is a commutative nuclear pomonoid (for $\alg{M}$ commutative),
\item $\CFreeRedUMon{\alg{M}}$\,is a unital commutative nuclear pomonoid (for $\alg{M}$ \mbox{commutative}).
\end{enumerate}
  In each case the nuclear image is isomorphic to $\alg{M}$ via the map $a \mapsto [a]$.
\end{theorem}

  Each of these pomonoids is \emph{(upper) product distributive} in the following sense:
\begin{align*}
  \text{if $u \sqleq v \circ w$, then $u_{1} \sqleq v$ and $u_{2} \sqleq w$ for some decomposition $u_{1} \circ u_{2} = u$.}
\end{align*}
  This turns out to be a crucial property of free nuclear preimages.\footnote{This definition, or more precisely its order dual, is due to Varlet~\cite{varlet69}. It should remind the reader of the definition of a distributive join semilattice.}

  Each of these constructions extends to a functor. The \emph{free nuclear preimage functor} from the category $\Pomon$ of pomonoids and their homomorphisms to the category $\PomonStar$ of nuclear pomonoids and their homomorphisms maps a pomonoid $\alg{M}$ to the nuclear pomonoid $\FreeMon{\alg{M}}$ and a homomorphism of pomonoids $h\colon \alg{M} \to \alg{N}$ to the map $\FreeMon{h}\colon \FreeMon{\alg{M}} \to \FreeMon{\alg{N}}$ such that
\begin{align*}
  \FreeMon{h} (\emptyword) & = \emptyword, \\
  \FreeMon{h} ([a_{1}, \dots, a_{n}]) & = [h(a_{1}), \dots, h(a_{n})].
\end{align*}
  Recall that $[a_{1}, \dots, a_{n}]$ and $[h(a_{1}), \dots, h(a_{n})]$ really represent equivalence classes of words here, so we need to verify that $\FreeMon{h}$ is well-defined. But indeed applying the map $h$ to a decomposition witnessing that $[a_{1}, \dots, a_{n}] \sqleq [b_{1}, \dots, b_{m}]$ yields a decomposition witnessing that $[h(a_{1}), \dots, h(a_{n})] \sqleq [h(b_{1}), \dots, h(b_{m})]$.

  The free \emph{unital} nuclear preimage functor maps $\alg{M}$ to $\FreeUMon{\alg{M}}$ instead of $\FreeMon{\alg{M}}$ and it maps $h\colon \alg{M} \to \alg{N}$ to the restriction of $\FreeMon{h}$ to $\FreeUMon{\alg{M}}$. Replacing $\FreeMon{\alg{M}}$ and $\FreeUMon{\alg{M}}$ by $\CFreeRedMon{\alg{M}}$ and $\CFreeRedUMon{\alg{M}}$ where $\alg{M}$ ranges over all commutative pomonoids yields the free \emph{commutative} nuclear preimage functor and the free \emph{commutative unital} nuclear preimage functor. 

  Each of these functors is the left adjoint of the nuclear image functor between the appropriate categories. For nuclear pomonoids the unit of this adjunction is the \mbox{natural} transformation $\unitmap$ consisting of the maps $\unit{\alg{M}}\colon \alg{M} \to (\FreeMon{\alg{M}})_{[\gamma]}$ such that $\unit{\alg{M}}\colon a \mapsto [a]$. The counit\footnote{The counit of this adjunction and the empty word are denoted by the same symbol. There is no threat of confusion, since they occur in completely different contexts. More\-over, we will seldom have occasion to talk about the counit of this adjunction.} is the natural transformation $\counitmap$ consisting of the maps $\counit{\pair{\alg{N}}{\delta}}\colon \FreeMon{(\alg{N}_{\delta})} \to \pair{\alg{N}}{\delta}$ such that
\begin{align*}
  \counit{\pair{\alg{N}}{\delta}}(\emptyword) & = \1, \\
  \counit{\pair{\alg{N}}{\delta}}([a_{1}, \dots, a_{n}]) & = a_{1} \cdot \ldots \cdot a_{n}.
\end{align*}
  In the unital case we replace $\FreeMon{\alg{M}}$ and $\FreeMon{(\alg{N}_{\delta})}$ by $\FreeUMon{\alg{M}}$ and $\FreeUMon{(\alg{N}_{\delta})}$ and we replace the maps $\unitmap$ and $\counitmap$ by their restrictions. In the commutative cases we replace the free monoids by the free commutative monoids.

\begin{theorem}
  The free (unital) nuclear preimage functor is left adjoint to the (unital) nuclear image functor. The unit and counit are $\unitmap$~and~$\counitmap$. The same holds for the free commutative (unital) nuclear preimage and image functors.
\end{theorem}

\begin{proof}
  It is easy to verify that these are indeed functors and that the maps $\unitmap$ and $\counitmap$ are natural transformations. We therefore only verify the triangle identities. If $\alg{M}$ is a pomonoid with a nucleus $\delta$ and $\gamma$ is the nucleus of~$\FreeMon{(\alg{M}_{\delta})}$, then $\counit{\pair{\alg{M}_{\delta}}{\gamma}} \circ \unit{\alg{M}_{\delta}} (a) = \counit{\pair{\alg{M}_{\delta}}{\gamma}} ([a]) = a$. On the other hand, if $\alg{M}$ is a pomonoid and $\gamma$ is the nucleus of $\FreeMon{\alg{M}}$, then for each word $[a_{1}, \dots, a_{n}] \in \FreeMon{\alg{M}}$
\begin{align*}
  \counit{\pair{\FreeMon{\alg{M}}}{\gamma}} \circ \unit{\FreeMon{\alg{M}}} ([a_{1}, \dots, a_{n}]) & = \counit{\pair{\FreeMon{\alg{M}}}{\gamma}} ([[a_{1}], \dots, [a_{n}]]) \\ & = [a_{1}] \circ \dots \circ [a_{n}] = [a_{1}, \dots, a_{n}]. \qedhere
\end{align*}
\end{proof}

  The above constructions work equally well for posemigroups instead of po\-monoids. Let $\alg{S} = \langle S, \leq, \cdot \rangle$ be a posemigroup and let $\FreeSem{\alg{S}} = \langle \FreeSem{S}, \circ \rangle$ be the free semigroup generated by~$\alg{S}$. Elements of $\FreeSem{\alg{S}}$ are thus non-empty words over~$\alg{S}$. This free semigroup is equipped with a semigroup homomorphism $\gamma\colon \FreeSem{\alg{S}} \to \alg{S}$ defined~as:
\begin{align*}
  \gamma ([a_{1} \cdots a_{n}]) & = a_{1} \cdot \ldots \cdot a_{n}.
\end{align*}
  The map $[\gamma]\colon \FreeSem{\alg{S}} \to \FreeSem{\alg{S}}$ is then defined as $[\gamma](u) = [\gamma(u)]$. The preorder $\sqleq$ on~$\FreeSem{S}$ is defined by the same condition as before:
\begin{align*}
  u \sqleq [a_{1} \dots a_{n}] \iff & \text{there are } u_{1}, \dots, u_{n} \in \FreeSem{\alg{S}} \text{ such that } \\
  & u = u_{1} \circ \dots \circ u_{n} \text{ and } \gamma(u_{i}) \leq a_{i} \text{ for each } u_{i}.
\end{align*}

  There is, however, an important difference between the posemigroup case and the pomonoidal case: in $\FreeSem{\alg{S}}$ we only admit non-empty words in the decomposition of $u$. We therefore obtain a substantially different \mbox{structure} even if $\alg{S}$ happens to be the semigroup reduct of a monoid $\alg{M}$. In particular, $\FreeSem{\alg{S}}$ never has a multiplicative unit. Moreover, the following facts hold for $\FreeSem{\alg{S}}$.

\begin{fact}
  If $[a_{1}, \dots, a_{m}] \sqleq [b_{1}, \dots, b_{n}]$ in $\FreeSem{\alg{S}}$, then $m \geq n$. 
\end{fact}

\begin{fact}
  $[a_{1}, \dots, a_{n}] \sqleq [b_{1}, \dots, b_{n}]$ in $\FreeSem{\alg{S}}$ if and only if $a_{i} \leq b_{i}$ for all $1 \leq i \leq n$.
\end{fact}

  The preorder of $\FreeSem{\alg{S}}$ is therefore always antisymmetric, while the preorder of~$\FreeMon{\alg{M}}$ is never antisymmetric. In other words, there is no need to take a partially ordered quotient of $\FreeSem{\alg{S}}$. This difference notwithstanding, the above proofs for $\FreeMon{\alg{M}}$ go through for $\FreeSem{\alg{S}}$ if we drop all reference to multiplicative units.

  If~the posemigroup $\alg{S}$ is commutative, we may replace the free semigroup in the definition of $\FreeSem{\alg{S}}$ by the free commutative semigroup. This results in the commutative nuclear posemigroup $\CFreeSem{\alg{S}}$. The following facts again imply that the preorder on $\CFreeSem{\alg{S}}$ is always antisymmetric.

\begin{fact}
  If $[a_{1}, \dots, a_{m}] \sqleq [b_{1}, \dots, b_{n}]$ in $\CFreeSem{\alg{S}}$, then $m \geq n$.
\end{fact}

\begin{fact}
  $[a_{1}, \dots, a_{n}] \sqleq [b_{1}, \dots, b_{n}]$ in $\CFreeSem{\alg{S}}$ if and only if $a_{i} \leq c_{i}$ for all $1 \leq i \leq n$ for some permutation $[c_{1}, \dots, c_{n}]$ of $[b_{1}, \dots, b_{n}]$.
\end{fact}

  The free nuclear preimage for posemigroups maps a posemigroup $\alg{S}$ to $\FreeSem{\alg{S}}$ and maps a homomorphism of posemigroups $h\colon \alg{S} \to \alg{T}$ to $\FreeSem{h}\colon \FreeSem{\alg{S}} \to \FreeSem{\alg{T}}$ such that $\FreeSem{h}([a_{1}, \dots, a_{n}) = [h(a_{1}), \dots, h(a_{n})]$. The unit and counit are again defined as $\unit{\alg{S}}\colon a \mapsto [a]$ and $\counit{\langle \alg{T}, \delta \rangle}\colon [a_{1}, \dots, a_{n}] \mapsto a_{1} \cdot \ldots \cdot a_{n}$. In the commutative case we replace the free semigroups by the free commutative semigroups.

\begin{theorem}
  Let $\alg{S}$ be a (commutative) posemigroup. Then $\FreeSem{\alg{S}}$ ($\CFreeSem{\alg{S}}$) is a (commutative) nuclear posemigroup. Its nuclear image is isomorphic to $\alg{S}$ via the map~$a \mapsto [a]$.
\end{theorem}

\begin{theorem}
  The free nuclear preimage functor on (commutative) posemigroups is left adjoint to the nuclear image functor on (commutative) nuclear posemigroups. The unit and counit are $\unitmap$ and $\counitmap$.
\end{theorem}

  The construction described in this section bears some similarity to the construction of Baer~\cite{baer49}, which Dvure\v{c}enskij and Vetterlein~\cite{dvurecenskij+vetterlein01b} used in their investigation of pseudoeffect algebras and pseudo MV-algebras. Baer\label{page: baer} in effect embeds a partial semigroup into a total semigroup which is free with respect to preserving all existing equalities of the form $a \cdot b = c$, whereas we embed a total posemigroup into another total posemigroup which is free with respect to preserving all inequalities of the form $a \cdot b \leq c$. The construction of Baer allows for finer control over products, at the cost of having to rely on divisibility to define a partial order from the semigroup structure. The construction of the free nuclear preimage does not rely on divisibility, but it gives up virtually all control over products: no equalities of the form $a \cdot b = c$ are preserved in the integral case.

  In the remainder of this section, we state some facts which will not be needed in the following, but which are nonetheless worth observing: $\FreeMon{\alg{M}}$ ($\CFreeRedMon{\alg{M}}$) always satisfies a limited form of cancellativity, but it almost never satisfies divisibility. We also describe the free unital nuclear preimage of the multiplicative reduct of the two-element Boolean algebra and the three-element {\L}ukasiewicz chain.

\begin{fact}
  If $w \sqleq u \circ w$ in $\FreeMon{\alg{M}}$ ($\CFreeRedMon{\alg{M}}$ for $\alg{M}$ commutative), then~$\emptyword \sqleq u$.
\end{fact}

\begin{proof}
  We prove the claim by induction over the sum of lengths of $u$ and $w$. It holds trivially if $u = \emptyword$ or $w = \emptyword$. Suppose therefore that $u = [a] \circ u'$. Then $w_{1} \sqleq [a]$ and $w_{2} \sqleq u' \circ w = u' \circ w_{1} \circ w_{2}$ for some $w_{1} \circ w_{2} = w$. If $w_{1} = \emptyword$, then $\emptyword \sqleq [a]$ and $w_{2} \sqleq u' \circ \emptyword \circ w_{2} \sqleq u' \circ w_{2}$, hence $\emptyword \sqleq u'$ by induction and $\emptyword = \emptyword \circ \emptyword \sqleq [a] \circ u' = u$. On the other hand, if $w_{1}$ is non-empty, then $w_{2} \sqleq u' \circ w_{1} \circ w_{2}$ implies $\emptyword \sqleq u' \circ w_{1}$ by induction, so $\emptyword \sqleq u'$ and $\emptyword \sqleq w_{1}$. Since $w_{1} \sqleq [a]$, this implies that $\emptyword \sqleq [a]$ and $\emptyword \sqleq [a] \circ u' = u$.
\end{proof}

  We shall call a pomonoid $\alg{M}$ \emph{divisible} if $a \leq b$ for $a, b \in \alg{M}$ implies that $b \cdot c = a = d \cdot b$ for some $c, d \in \alg{M}$.\footnote{Often, the further condition $c, d \leq \1$ is required in the definition of divisibility. Since our goal is to show that free unital nuclear preimages are almost never divisible, we use the weaker definition.} The pomonoid $\FreeMon{\alg{M}}$ ($\CFreeMon{\alg{M}}$ for $\alg{M}$ commutative) is never divisible, on account of the empty word.

\begin{fact}
  $\FreeUMon{\alg{M}}$ is divisible if and only if the partial order of $\alg{M}$ is the equality relation or $\alg{M}$ is the two-element meet semilattice. $\CFreeUMon{\alg{M}}$ for a commutative pomonoid $\alg{M}$ is divisible if and only if the partial order of $\alg{M}$ is the equality relation, except possibly for a single comparable pair $a < \1$.
\end{fact}

\begin{proof}
  We first show that if $\FreeUMon{\alg{M}}$ or $\CFreeUMon{\alg{M}}$ is divisible, then $\1$ is a maximal element of $\alg{M}$. If $\1$ is not maximal in $\alg{M}$, then $\1 < a$ for some $a \in \alg{M}$, hence by divisibility $[\1] = u \circ [a]$ in $\FreeUMon{\alg{M}}$ for some word $u$. But $\emptyword \sqleq [\1] \sqleq u \circ [a]$ implies that $\emptyword \sqleq u$, so $[a] \sqleq \emptyword \circ [a] \sqleq u \circ [a] \sqleq [\1]$ and $a \leq \1$, contradicting $\1 < a$.

  Secondly, we show that if $\FreeUMon{\alg{M}}$ or $\CFreeUMon{\alg{M}}$ is divisible, then $a < b$ in $\alg{M}$ implies that $\1 \leq b$ and $a \cdot b = a = b \cdot a$. If $a < b$, then $[a] \sqleq [b]$, so $[a] \sqleq u \circ [b]$ and $u \circ [b] \sqleq [a]$ for some word $u$. Then either $\emptyword \sqleq u$ and $[a] \sqleq [b]$, or $[a] \sqleq u_{1}$ and $\emptyword \sqleq u_{2} \circ [b]$ for some decomposition $u_{1} \circ u_{2} = u$. In the former case, $u \circ [b] \sqleq [a]$ implies $[b] \sqleq [a]$, which contradicts $a < b$. In the latter case, $\emptyword \sqleq u_{2}$ and $\emptyword \sqleq [b]$, so $[a] \circ [b] \sqleq u_{1} \circ [b] \sqleq u_{1} \circ u_{2} \circ [b] \sqleq [a]$. Thus $a \cdot b \leq a$ and $\1 \leq b$, hence also $a \leq a \cdot b$ and $a \cdot b = b$. Similarly, we obtain that $b \cdot a = b$.

  It follows that every pair of elements $\alg{M}$ not equal to $\1$ is incomparable, and every element other than $\1$ is either incomparable to $\1$ or strictly below $\1$. If $a < \1$ and $b < \1$ for $a \neq b$, then $a \cdot b \leq a$ and $a \cdot b \leq b$, so $a = a \cdot b = b$. The partial order of $\alg{M}$ is therefore the equality relation, except possibly for a single element $a$ such that $a < 1$. Suppose that such an element $a$ exists. Then $u \sqleq v$ for non-empty words $u$, $v$ if and only if $v$ is obtained from $u$ by removing some instances of $[a]$ or $[\1]$. In the commutative case, this results in a divisible order. In the non-commutative case, $[b] \circ [a] \sqleq [b]$ for each $b \in \alg{M}$, so $[b] \circ [a] = u \circ [b]$ for some non-empty word $u$. If $b \neq \1$, it follows that $u = [a] = [b]$.
\end{proof}

  Finally, let us explicitly describe two of the simplest examples of free (commutative) unital nuclear preimages. If $\alg{M}$ is the two-element unital meet semilattice, then $\FreeUMon{\alg{M}} = \CFreeUMon{\alg{M}}$ is a chain of the form
\begin{align*}
  \cdots < [a] \circ [a] \circ [a] < [a] \circ [a] < [a] < [\1].
\end{align*}
  If $\alg{M}$ is the multiplicative reduct of the three-element {\L}ukasiewicz chain, i.e. the chain $b < a < \1$ with $a \cdot a = b$ (hence $x \cdot b = b = b \cdot x$ for each $x \in \alg{M}$), then every element of $\CFreeUMon{\alg{M}}$ has the form $[a]^{i} \circ [b]^{j}$ for some $i, j \in \omega$. The monoid reduct of $\CFreeUMon{\alg{M}}$ is therefore isomorphic to $\mathbb{N}^{2}$, where $\mathbb{N}$ is the additive monoid of non-negative integers. The partial order on $\CFreeUMon{\alg{M}}$ is then defined by the following equivalence:
\begin{align*}
  [a]^{i} \circ [b]^{j} \sqleq [a]^{m} \circ [b]^{n} \iff m \leq i \text{ and } 2n \leq 2j + (i-m),
\end{align*}
  which essentially states that two instances of $[a]$ may be traded for a single instance of $[b]$ when going up in the order. As for $\FreeUMon{\alg{M}}$, already for this three-element pomonoid it does not seem to admit a description which would be substantially simpler than just restating the definition of $\FreeUMon{\alg{M}}$. That is, one obtains a partial order on the set of words on a two-element alphabet consisting of $a, b$ such that when going up, one can either erase a letter, or trade a single $[b]$ for a single $[a]$, or trade two instances of $[a]$ for a single $[b]$.

\section{Free nuclear preimages of \texorpdfstring{s$\ell$-monoids}{sl-monoids}}

  Like the free partially ordered nuclear preimage functor, the free \emph{semilattice-ordered} nuclear preimage functor is the left adjoint of a nuclear image functor, namely the functor from the category of nuclear s$\ell$-monoids $\SlmonStar$ to the \mbox{category} of s$\ell$-monoids $\Slmon$. We show that the free semilattice-ordered nuclear preimage can be obtained from the free partially ordered nuclear preimage as the s$\ell$-monoid of non-empty finitely generated downsets. This relies on the observation that the free partially ordered nuclear pre\-image of an s$\ell$-monoid belongs to a category inter\-mediate between nuclear pomonoids and nuclear s$\ell$-monoids.

  A nucleus $\gamma$ on a pomonoid $\alg{M}$ will be called an \emph{s$\ell$-nucleus} if its nuclear \mbox{image} $\alg{M}_{\gamma}$ is a join semi\-lattice. We use $x \veegamma y$ for $x, y \in \alg{M}_{\gamma}$ to denote joins in $\alg{M}_{\gamma}$. This extends to a total operation on $\alg{M}$: $x \veegamma y \assign \gamma(x) \veegamma \gamma(y)$ for arbitrary $x, y \in \alg{M}$. (Note that $x \veegamma y$ should not be interpreted as $\gamma (x \vee y)$ here, since $x \vee y$ need not exist.) An \emph{s$\ell$-nuclear pomonoid} is a pomonoid $\alg{M}$ equipped with an s$\ell$-nucleus $\gamma$ and with the operation $x \veegamma y$. Their homomorphisms are homomorphisms of nuclear pomonoids which preserve~$\veegamma$. The category of s$\ell$-nuclear pomonoids and their homomorphisms will be denoted $\PomonSlStar$.

  The nuclear image functor $\SlmonStar \to \Slmon$ is the composition of two functors: the forgetful functor $\SlmonStar \to \PomonSlStar$ and the nuclear image functor $\PomonSlStar \to \Slmon$. To find the left adjoint of the nuclear image functor $\SlmonStar \to \Slmon$, it therefore suffices to find the left adjoints of these two functors and compose them.

  The left adjoint of the nuclear image functor $\PomonSlStar \to \Slmon$ co\-incides with the restriction to s$\ell$-monoids of the free nuclear preimage functor for pomonoids. This is because the free partially ordered nuclear preimage functor sends (homomorphisms of) s$\ell$-monoids to (homomorphisms of) s$\ell$-nuclear pomonoids, and moreover if $\langle \alg{N}, \delta, \vee_{\delta} \rangle$ is an s$\ell$-nuclear pomonoid, then the counit $\counit{\langle \alg{N}, \delta \rangle}\colon \FreeMon{\alg{N}_{\delta}} \to \langle \alg{N}, \delta \rangle$ is a homomorphism of s$\ell$-nuclear pomonoids. \emph{Mutatis mutandis} the same holds for the unital, commutative, and unital commutative free nuclear preimages.  

  It therefore only remains to find the left adjoint of the forgetful functor from nuclear s$\ell$-monoids to s$\ell$-nuclear pomonoids. Let us start with the more modest task of describing the left adjoint of the forgetful functor from s$\ell$-monoids to pomonoids.

  A downset $X$ of a poset $P$ is called \emph{finitely generated} if it is the downward closure of a finite set of points. That is, $X = \below \{ x_{1}, \dots, x_{m} \}$, where $\below A$ denotes the downward closure of a set $A \subseteq P$. The non-empty finitely generated downsets of $P$ yield a join semilattice $\Id P$ where binary joins are unions.

  It is common to call a join semilattice distributive \label{page: distributive} if $a \leq b \vee c$ implies that there are $b' \leq b$ and $c' \leq c$ such that $a = b' \vee c'$. This definition, however, forces the semilattice to be down-directed. For semilattices which need not be down-directed, the following definition of Katri\v{n}\'{a}k~\cite{katrinak68} is more appropriate. We call a join semilattice \emph{distributive} if $a \leq b \vee c$ implies that either $a \leq b$ or $a \leq c$ or there are $b' \leq b$ and $c' \leq c$ such that $a = b' \vee c'$. The relation between the two definitions is simple: a join semilattice is distributive in the stronger sense if and only if it is down-directed and distributive in the weaker sense. If the semilattice is a lattice, both of these definitions agree with the usual definition of a distributive lattice.

\begin{fact}
  $\Id P$ is a distributive semilattice.
\end{fact}

  The non-empty finitely generated downsets of a (commutative) pomonoid $\alg{M}$ moreover form a (commutative) s$\ell$-monoid $\Id \alg{M}$ with the multiplication
\begin{align*}
  X \ast Y & \assign \below (X \cdot Y).
\end{align*}
  More explicitly, multiplication in $\Id {\alg{M}}$ is defined as
\begin{align*}
  \below \{ x_1, \dots, x_m \} \ast \below \{ y_1, \dots, y_n \} \assign \below \set{x_{i} y_{j}}{1 \leq i \leq m, 1 \leq j \leq n},
\end{align*}
  If $\alg{M}$ has a multiplicative unit $\1$, then $\Id \alg{M}$ has a multiplicative unit $\below \1$, where $\below a \assign \below \{ a \}$. This map $\below$ embeds the pomonoid $\alg{M}$ into the s$\ell$-monoid $\Id \alg{M}$. In~fact, it exhibits $\Id \alg{M}$ as the free s$\ell$-monoid over the pomonoid $\alg{M}$.

\begin{fact}
  $\Id \alg{M}$ is the free s$\ell$-monoid over the pomonoid $\alg{M}$.
\end{fact}

\begin{proof}
  The unit of the adjunction (between the category of pomonoids and the category of s$\ell$-monoids) is the map $a \mapsto \below a$. The counit is the map $\below \{ x_{1}, \dots, x_{n} \} \mapsto x_{1} \vee \dots \vee x_{n}$. Verifying the triangle identities is straight\-forward.
\end{proof}

\begin{fact}
  Let $\gamma$ be an s$\ell$-nucleus on a pomonoid $\alg{M}$. Then the following is a nucleus on the s$\ell$-monoid~$\Id \alg{M}$:
\begin{align*}
  \gammaid(\below \{ x_1, \dots, x_n \}) & \assign \below (\gamma(x_{1}) \veegamma \dots \veegamma \gamma(x_{n})).
\end{align*}
\end{fact}

\begin{proof}
  Let $a \assign \below \{ x_1, \dots, x_m \}$ and $b \assign \below \{ y_1, \dots, y_n \}$. Then
\begin{align*}
  \gamma(a) \cdot \gamma(b) & = \below (\gamma(x_1) \veegamma \dots \veegamma \gamma(x_m)) \cdot \below (\gamma(y_1) \veegamma \dots \veegamma \gamma(y_n)) \\
  & = \below [(\gamma(x_1) \veegamma \dots \veegamma \gamma(x_m)) \cdot (\gamma(y_1) \veegamma \dots \veegamma \gamma(y_n))] \\
  & \leq \below [(\gamma(x_1) \veegamma \dots \veegamma \gamma(x_m)) \cdotgamma (\gamma(y_1) \veegamma \dots \veegamma \gamma(y_n))] \\
  & = \below \bigveegamma \set{\gamma(x_i) \cdotgamma \gamma(y_j)}{1 \leq i \leq m \text{ and } 1 \leq j \leq n}.
\end{align*}
  Let us abbreviate this final join by $z$. Then $z \leq \gamma(a \cdot b)$ if and only if $\gamma(x_i) \cdotgamma \gamma(y_j) \leq \gamma(a \cdot b)$ for each $x_i$ and $y_j$, or equivalently $\gamma(x_i) \cdot \gamma(y_j) \leq \gamma(a \cdot b)$ for each $x_i$ and $y_j$. But $\gamma(x_i) \cdot \gamma(y_j) \leq \gamma (x_i \cdot y_j) \leq \gamma(a \cdot b)$.
\end{proof}

\begin{fact}
  Let $\gamma$ be a (unital) s$\ell$-nucleus on a pomonoid $\alg{M}$. Then $\gammaid$ is a (unital) nucleus on~the s$\ell$-monoid $\Id \alg{M}$.
\end{fact}

\begin{theorem}
  Let $\alg{M}$ be an s$\ell$-monoid. Then
\begin{enumerate}[(i)]
\item $\Id \FreeMon{\alg{M}}$ is a nuclear s$\ell$-monoid,
\item $\Id \FreeUMon{\alg{M}}$ is a unital nuclear s$\ell$-monoid,
\item $\Id \CFreeRedMon{\alg{M}}$ is a commutative nuclear s$\ell$-monoid (for $\alg{M}$ commutative),
\item $\Id \CFreeRedUMon{\alg{M}}$\,is a unital commutative nuclear s$\ell$-monoid (for $\alg{M}$ \mbox{commutative}).
\end{enumerate}
  In each case the nuclear image is isomorphic to $\alg{M}$ via the map $a \mapsto \below [a]$.
\end{theorem}

  The above construction is in fact functorial. Given a homomorphism of s$\ell$-nuclear pomonoids ${h\colon \alg{M} \to \alg{N}}$, we define $\Id h\colon \Id \alg{M} \to \Id \alg{N}$ to be the homomorphism of nuclear s$\ell$-monoids
\begin{align*}
  \Id h\colon \below \{ x_{1}, \dots, x_{n} \} \mapsto \below \{ h(x_{1}), \dots, h(x_{n}) \}. 
\end{align*}
  (The requirement that homomorphisms of s$\ell$-nuclear pomonoids must preserve the operation $x \veegamma y$ is needed here.) This yields a functor $\Id$ from the category of s$\ell$-nuclear monoids $\PomonSlStar$ to the category of nuclear s$\ell$-monoids $\SlmonStar$.

  This functor is left adjoint to the forgetful functor $\SlmonStar \to \PomonSlStar$. The unit of this adjunction consists of the maps $\unit{\alg{M}}\colon \langle \alg{M}, \gamma \rangle \to \langle \Id \alg{S}, \gammaid \rangle$ such that $\unit{\alg{M}}(a) = \below a$, while the counit consists of the maps $\counit{\langle \alg{S}, \gamma \rangle}\colon \langle \Id \alg{M}, \gammaid \rangle \to \langle \alg{M}, \gamma \rangle$ such that
\begin{align*}
  \counit{\langle \alg{M}, \gamma \rangle}(\below \{ x_1, \dots, x_n \}) & \assign x_1 \vee \dots \vee x_n.
\end{align*}

\begin{fact}
  The functor $\Id$ on s$\ell$-nuclear pomonoids is left adjoint to the forget\-ful functor from $\SlmonStar$ to $\PomonSlStar$. The~unit and counit are $\unitmap$ and $\counitmap$.
\end{fact}

\begin{proof}
  The proof is routine. We only verify that $\counitmap$ indeed preserves the nucleus:
\begin{align*}
  \counit{\pair{\alg{S}}{\gamma}}(\gammaid \below \{ x_1, \dots, x_n \}) & = \counit{\pair{\alg{S}}{\gamma}}(\below (\gamma(x_1) \veegamma \dots \veegamma \gamma(x_n))) \\ & = \gamma(x_1) \veegamma \dots \veegamma \gamma(x_n) \\ & = \gamma(x_1 \vee \dots \vee x_n) \\ & = \gamma (\counit{\pair{\alg{S}}{\gamma}} (\below \{ x_1, \dots, x_n \})). \qedhere
\end{align*}
\end{proof}

  Putting the above observations together yields the following theorem.

\begingroup
\tolerance=400

\begin{theorem}
  The free semilattice-ordered (unital) nuclear preimage of an s$\ell$-monoid $\alg{M}$ is $\Id \FreeMon{\alg{M}}$ ($\Id \FreeUMon{\alg{M}})$. Similarly, the free commutative semilattice-ordered (unit\-al) nuclear preimage of a commutative s$\ell$-monoid $\alg{M}$ is $\Id \CFreeRedMon{\alg{M}}$ ($\Id \CFreeRedUMon{\alg{M}})$.
\end{theorem}

\endgroup

\section{Nuclear images and simple quasi-inequations}

  The properties of the free nuclear preimage reflect in a very direct way which pomonoidal quasi-inequations are preserved under nuclear images. We devote this section to finding a syntactic description of such quasi-inequations.

  Let us first review some preliminaries concerning ordered algebras. Consider a fixed algebraic signature and an infinite set of variables $\Var$. The term algebra, i.e.\ the absolutely free algebra over $\Var$, will be denoted $\Tm$. An \emph{interpretation} on an algebra $\alg{A}$ is a homomorphism $h\colon \Tm \to \alg{A}$. If the range of $h$ is a subset of some $X \subseteq \alg{A}$, we say that $h$ is an interpretation \emph{over} $X$. Each map $f\colon \Var \to \alg{A}$ extends to a unique interpretation $h\colon \Tm \to \alg{A}$.

  An \emph{ordered algebra} is an algebra equipped with a partial order. (One may bake monotonicity conditions for each operation into the definition of an ordered algebra, but for our purposes this is not necessary.) An \emph{ordered quasivariety} is a class of ordered algebras axiomatized by a set of \emph{quasi-inequations}, i.e.\ universally quantified sentences of the form
\begin{align*}
  t_{1} \leq u_{1} ~ \& ~ \dots ~ \& ~ t_{n} \leq u_{n} & \implies t \leq u,
\end{align*}
  where $t_{1}, \dots, t_{n}, u_{1}, \dots, u_{n}, t, u$ are terms in the appropriate algebraic signature. Equivalently, an ordered quasivariety is a class of ordered algebras which is closed under isomorphic images, substructures, products, and ultraproducts.

  An \emph{order congruence} on the ordered algebra $\alg{A}$ ordered by $\leq_{\alg{A}}$ is a preorder $\preleq$ which extends the partial order ${\leq_{\alg{A}}}$ such that the equivalence relation $\sim$ corresponding to $\preleq$ is a congruence of $\alg{A}$. (If we were to impose some monotonicity conditions in the definition of an ordered algebra, then we would also need to require that the preorder satisfy these monotonicity conditions.) In~that case $\alg{A} / {\preleq}$ denotes the algebra $\alg{A} / {\sim}$ ordered by $\leqK$, i.e.\ $[a]_{\sim} \leq [b]_{\sim}$ in $\alg{A} / {\sim}$ if and only if $a \leqK b$.

  Order congruences, ordered quasivarieties, and pomonoids are related in much the same way that congruences, quasivarieties, and monoids are. If $\class{K}$ is an ordered quasivariety, a \emph{$\class{K}$-congruence} on $\alg{A}$ is an ordered congruence ${\leqK}$ on $\alg{A}$ such that $\alg{A} / {\leqK} \in \class{K}$. Each algebra $\alg{A}$ has a smallest $\class{K}$-congruence. If $h\colon \alg{A} \to \alg{B} \in \class{K}$ is a surjective monotone homomorphism of ordered algebras, then ${\preleq} \assign h^{-1}[\leq_{\alg{B}}]$ is a $\class{K}$-congruence on $\alg{A}$ and $\alg{B} \iso \alg{A} / {\preleq}$. The following lemma is almost a special case of a lemma of Blok and Raftery~\cite[Lemma~4.2]{blok+raftery99}, except that they consider ordinary quasivarieties rather than ordered quasivarieties.

\begin{lemma}
  Let $\class{K}$ be an ordered quasivariety and ${\leqK}$ the smallest $\class{K}$-congruence of $\alg{M}$. Then $a \leqK b$ if and only if there is a quasi-inequation $\qine$ valid in $\class{K}$ and an interpretation $h$ on $\alg{M}$ such that $h(t_{i}) \leq h(u_{i})$ in $\alg{M}$ for each premise $t_{i} \leq u_{i}$ of $\qine$ and moreover $h(t) = a$ and $h(u) = b$ for the conclusion $t \leq u$ of $\qine$.
\end{lemma}

\begin{proof}
  Clearly if $a$ and $b$ satisfy this condition, then $a \leqK b$. Conversely, let $\sqleq$ be the relation defined by this condition. If $\sqleq$ is a $\class{K}$-congruence on $\alg{A}$, then it must be the smallest $\class{K}$-congruence on~$\alg{A}$.

  To prove that $\sqleq$ is reflexive, take $\alpha$ to be the inequality $x \leq x$. To prove transitivity, if the quasi-inequation $\alpha \assign E \implies t \leq u$ witnesses $a \leq b$ and the quasi-inequation $\alpha' \assign E' \implies t' \leq u'$ witnesses $b \leq c$ and these quasi-inequations without loss of generality do not share any variables, then the quasi-inequation $E ~ \& ~ E' ~ \& ~ u \leq t' \implies t \leq u'$ witnesses $a \leq c$. The fact that $\sqleq$ is an order congruence follows by combining quasi-inequations in a similar way. Finally, suppose that $\alpha \assign t_{1} \leq u_{1} ~ \& ~ \dots ~ \& ~ t_{n} \leq u_{n} \implies t \leq u$ is a quasi-inequation valid in $\class{K}$. To prove that $\sqleq$ is a $\class{K}$-congruence, we must prove that for each interpretation $h$ on $\alg{A}$ we have $h(t) \sqleq h(u)$ whenever $h(t_{i}) \sqleq h(u_{i})$ for $1 \leq i \leq n$. Each of the inequalities $h(t_{i}) \sqleq h(u_{i})$ is witnessed by some quasi-inequation $E'_{i} \implies t'_{i} \leq u'_{i}$. We may assume without loss of generality that the $n+1$ quasi-inequations $\alpha$ and $\alpha'_{i}$ do not share any variables. Then the quasi-inequation $E'' \implies t \leq u$ is valid in $\class{K}$ and witnesses $h(t_{i}) \sqleq h(u_{i})$, where $E''$ is the conjunction of the inequalities $E'_{i}$, $t_{i} \leq t'_{i}$, $u'_{i} \leq u_{i}$ for $1 \leq i \leq n$. 
\end{proof}

  We say that the quasi-inequation $\qine$ in the above lemma \emph{witnesses} that $a \leqK b$ with respect to the inter\-pretation~$h$. We say that $\leqK$ is witnessed by a certain class of quasi-inequations with respect to a certain class of interpretations if for each $a \leqK b$ we may choose $\alpha$ and $h$ from these classes.

  Returning to the ordered quasivariety of nuclear pomonoids (s$\ell$-monoids), a quasi-inequation in the signature of nuclear pomonoids (s$\ell$-monoids) will be called \emph{simple} if in each premise $t_{i} \leq u_{i}$ the term $u_{i}$ is either a variable $x_{i}$ or it has the form $u_{i} = \gamma(v_{i})$ for some term $v_{i}$. The \emph{simple quasi-inequational theory} of a class $\class{K}$ of nuclear pomonoids (s$\ell$-monoids) is the set of all simple quasi-inequations valid in each ordered algebra in $\class{K}$.

  In the following, let $\alg{M}$ be a pomonoid and $\gamma$ be a nucleus on $\alg{M}$. We shall view the nuclear image $\alg{M}_{\gamma}$ as a nuclear pomonoid where the nucleus is the identity map.

\begin{fact} \label{fact: simple quasi-inequations}
  Each simple quasi-inequation valid in $\langle \alg{M}, \gamma \rangle$ also holds in $\alg{M}_{\gamma}$.
\end{fact}

\begin{proof}
  Firstly, observe that $t^{\alg{M}_{\gamma}}(\gamma (\tuple{a})) \leq \gamma(t^{\alg{M}}(\tuple{a}))$, where $\tuple{a} = \langle a_{1}, \dots, a_{k} \rangle$ is a tuple of elements of $\alg{M}$ and $\gamma(\tuple{a}) = \langle \gamma(a_{1}), \dots, \gamma(a_{k}) \rangle$. This is straightforward to prove by induction over the complexity of the term~$t$. Moreover, $\gamma(t^{\alg{M}}(\tuple{a})) = t^{\alg{M}_{\gamma}}(\tuple{a})$ for each tuple $\tuple{a}$ of elements of $\alg{M}_{\gamma}$.

  Now consider a premise $t_{i}(\tuple{x}) \leq u_{i}(\tuple{x})$ of a simple quasi-inequation~$\alpha$ which holds in~$\langle \alg{M}, \gamma \rangle$. If $t^{\alg{M}_{\gamma}}_{i}(\tuple{a}) \leq u^{\alg{M}_{\gamma}}_{i}(\tuple{a})$ for a tuple $\tuple{a}$ of elements of $\alg{M}_{\gamma}$, then $t^{\alg{M}}_{i}(\tuple{a}) \leq \gamma(t^{\alg{M}}_{i}(\tuple{a})) = t^{\alg{M}_{\gamma}}_{i}(\tuple{a}) \leq u^{\alg{M}_{\gamma}}_{i}(\tuple{a}) = u^{\alg{M}}_{i}(\tuple{a})$, where the last equality holds because either $u_{i} = x_{i}$ or $u_{i} = \gamma(v_{i})$. Each premise of $\alpha$ is thus satisfied in $\langle \alg{M}, \gamma \rangle$ by the tuple $\tuple{a}$, hence so is the conclusion $t(\tuple{x}) \leq u(\tuple{x})$, i.e.\ $t^{\alg{M}}(\tuple{a}) \leq u^{\alg{M}}(\tuple{a})$ holds in $\alg{M}$. It follows that $t^{\alg{M}_{\gamma}}(\tuple{a}) = \gamma(t^{\alg{M}}(\tuple{a})) \leq \gamma(u^{\alg{M}}(\tuple{a})) = u^{\alg{M}_{\gamma}}(\tuple{a})$. The quasi-inequation $\alpha$ therefore holds in $\alg{M}_{\gamma}$.
\end{proof}

  Let $\class{K}$ be an ordered quasivariety of nuclear pomonoids and let $\leqK$ be the smallest $\class{K}$-congruence on $\FreeMon{\alg{M}}$.

\begin{lemma} \label{lemma: restriction on free}
  $\alg{M}$ is the nuclear image of a nuclear pomonoid in $\class{K}$ if and only if $v \leqK w$ in $\FreeMon{\alg{M}}$ implies $\gamma(v) \leq \gamma(w)$ in $\alg{M}$.
\end{lemma}

\begin{proof}
  If $v \leqK w$ implies $\gamma(v) \leq \gamma(w)$, then the restriction of $\leqK$ to singleton words coincides with $\sqleq$. The nuclear image of $(\FreeMon{\alg{M}}) / {\leqK}$, which is a nuclear pomonoid in $\class{K}$, and therefore coincides with the nuclear image of $\FreeMon{\alg{M}}$, i.e.\ it is isomorphic to $\alg{M}$. Because $\class{K}$ is closed under isomorphisms, we may assume that this nuclear image is in fact equal to $\alg{M}$.

  Conversely, suppose that $\alg{M} = \alg{N}_{\delta}$ for some nuclear pomonoid ${\pair{\alg{N}}{\delta} \in \class{K}}$. There is a surjective homomorphism $\counit{\pair{\alg{N}}{\delta}} \colon \FreeMon{\alg{M}} \to \pair{\alg{N}}{\delta}$. If $v \leqK w$ in $\FreeMon{\alg{M}}$, then $\counit{\pair{\alg{N}}{\delta}}(v) \leq \counit{\pair{\alg{N}}{\delta}}(w)$ in $\alg{N}$, and therefore $\gamma(v) = \counit{\pair{\alg{N}}{\delta}}([\gamma(v)]) = \delta(\counit{\pair{\alg{N}}{\delta}}(v)) \leq \delta(\counit{\pair{\alg{N}}{\delta}}(v)) = \counit{\pair{\alg{N}}{\delta}}(\gamma(w)) = \gamma(w)$.
\end{proof}

\begin{lemma} \label{lemma: witnessed by simple}
  The $\class{K}$-congruence $\leqK$ is witnessed by simple quasi-inequations with respect to interpretations over singleton words.
\end{lemma}

\begin{proof}
  This holds because $\FreeMon{\alg{M}}$ is (upper) product distributive and generated as a monoid by singleton words. More explicitly, suppose that the quasi-inequation $\alpha$ valid in $\class{K}$ witnesses that $p \leqK q$ with respect to some inter\-pretation $h$. If $h(x) = [a_{1}, \dots, a_{n}]$ where $x$ is a variable, we substitute $x$ in~$\alpha$ by the term $x_{1} \cdot \ldots \cdot x_{n}$ and take $h(x_{i}) = a_{i}$. If $h(x) = \emptyword$, we substitute $x$ by the term~$\1$. This yields an interpretation over atomic words and a quasi-inequation valid in $\class{K}$ witnessing $p \leqK q$.

  Now consider some premise $t_{i} \leq u_{i}$ of $\alpha$. If the term $u_{i}$ has the form $u_{i} = u_{i,1} \cdot u_{i,2}$, then $h(t_{i}) \sqleq h(u_{i,1}) \circ h(u_{i,2})$, so by product distributivity there is a decomposition $h(t_{i}) = r_{i,1} \circ r_{i,2}$ such that $r_{i,1} \sqleq h(u_{i,1})$ and $r_{i,2} \sqleq h(u_{i,2})$. The term $t_{i}$ is a (possibly empty) product of variables and terms of the form $\gamma(v)$. Because $h$ is an interpretation over atomic words, there are terms $t_{i,1}$ and $t_{i,2}$ such that $t_{i} = t_{i,1} \cdot t_{i,2}$ holds in all nuclear pomonoids and moreover $h(t_{i,1}) = r_{i,1}$ and $h(t_{i,2}) = r_{i,2}$. It~follows that in the quasi-inequation $\alpha$ we may replace the premise $t_{i} \leq u_{i}$ by the two premises $t_{i,1} \leq u_{i,1}$ and $t_{i,2} \leq u_{i,2}$. This results in a new quasi-inequation $\beta$. Clearly $\beta$ is a consequence of $\alpha$, since $t_{i,1} \leq u_{i,1} ~ \& ~ t_{i,2} \leq u_{i,2} \implies t_{i} \leq u_{i}$ holds in all nuclear pomonoids. In particular, $\beta$ is valid in $\class{K}$. Moreover, $\beta$ still witnesses that $p \leqK q$ in $\FreeMon{\alg{M}}$ with respect to $h$.

  If the term $u_{i}$ has the form $u_{i} = \1$, then for each interpretation $h$ on $\FreeMon{\alg{M}}$ we have $h(t_{i}) \leq h(u_{i}) = \emptyword$ only if $h(t_{i}) = \emptyword$, which in turn only holds if either $t_{i} = \1$ (in which case the premise is redundant) or $t_{i}$ is a product of variables $t_{i} = x_{i,1} \cdot \ldots \cdot x_{i,k}$ and $h(x_{i,j}) = \emptyword$ for each of them. In that case, we may replace the premise $t_{i} \leq u_{i}$ by the premises $x_{i,1} \leq u_{i}$, \dots, $x_{i,k} \leq u_{i}$ to again obtain a new quasi-inequation $\beta$ valid in $\class{K}$ which witnesses $p \leqK q$ with respect to $h$.
\end{proof}

\begin{lemma} \label{lemma: witness transfer}
  Let $\qine$ be a simple quasi-inequation and $g$ be an interpretation on~$\alg{M}$. If $\qine$ witnesses that $v \leqK w$ in $\FreeMon{\alg{M}}$ with respect to the interpretation $h(x) \assign [g(x)]$, then $\qine$ also witnesses that $\gamma(v) \leq \gamma(w)$ in $\alg{M}$ with respect to $g$.
\end{lemma}

\begin{proof}
  Recall that $u \sqleq [a]$ in $\FreeMon{\alg{M}}$ if and only if $\gamma(u) \leq a$ in $\alg{M}$. Moreover, $\gamma(t([a_{1}], \dots, [a_{n}])) = t(a_{1}, \dots, a_{n})$ in $\alg{M}$. Consequently, $t([a_{1}], \dots, [a_{n}]) \leq [b]$ holds in $\FreeMon{\alg{M}}$ only if $t(a_{1}, \dots, a_{n}) \leq b$ holds in $\alg{M}$, therefore $g(t_{i}) \leq g(u_{i})$ holds in $\alg{M}$ for each premise $t_{i} \leq u_{i}$ of $\qine$, and $g(t) = \gamma(h(t)) = \gamma(v)$ and $g(u) = \gamma(h(u)) = \gamma(w)$ hold for the conclusion $t \leq u$ of $\qine$.
\end{proof}

  We remind the reader that in the following two theorems we identify pomonoids (s$\ell$-monoids) with nuclear pomonoids (s$\ell$-monoids) where $\gamma(x) = x$ for each $x$.

\begin{theorem} 
  Let $\class{K}$ be an ordered quasivariety of nuclear pomonoids. Then $\class{K}_{\gamma} \assign \set{\alg{M}_{\gamma}}{\langle \alg{M}, \gamma \rangle \in \class{K}}$ is an ordered quasivariety of pomonoids axiomatized by the simple quasi-inequational theory of $\class{K}$.
\end{theorem}

\begin{proof}
  If $\langle \alg{M}, \gamma \rangle \in \class{K}$, then $\alg{M}_{\gamma}$ satisfies the simple quasi-inequational theory of~$\class{K}$ by Fact~\ref{fact: simple quasi-inequations}. Conversely, suppose that a pomonoid $\alg{M}$ satisfies the simple quasi-inequational theory of $\class{K}$. By Lemma~\ref{lemma: restriction on free} it suffices to show that $u \leqK v$ in $\FreeMon{\alg{M}}$ implies $\gamma(u) \leq \gamma(v)$ in $\alg{M}$, where ${\leqK}$ is the smallest $\class{K}$-congruence of $\FreeMon{\alg{M}}$. By Lemma~\ref{lemma: witnessed by simple}, if $u \leqK v$, then this is witnessed by some simple quasi-inequation $\qine$ with respect to a valuation $h\colon \Tm \to \FreeMon{\alg{M}}$ which assigns a singleton word to each variable. Then $\qine$ also witnesses that $\gamma(u) \leq \gamma(v)$ in~$\alg{M}$ by Lemma~\ref{lemma: witness transfer}.
\end{proof}

  An entirely analogous theorem holds for s$\ell$-monoids. The only change required in the proof consists in replacing the free pomonoidal \mbox{nuclear} preimage $\FreeMon{\alg{M}}$ by the free s$\ell$-monoidal nuclear preimage $\Id \FreeMon{\alg{M}}$ and using the distributivity of $\Id \FreeMon{\alg{M}}$ in a manner entirely analogous to the product distributivity of $\FreeMon{\alg{M}}$ to eliminate joins from the right-hand sides of the premises of simple quasi-inequations.

\begin{theorem} 
  Let $\class{K}$ be a quasivariety of nuclear s$\ell$-monoids. Then $\class{K}_{\gamma} \assign \set{\alg{M}_{\gamma}}{\langle \alg{M}, \gamma \rangle \in \class{K}}$ is a quasi\-variety of s$\ell$-monoids axiomatized by the simple quasi-inequational theory of $\class{K}$.
\end{theorem}

  The above theorems, however, do not directly provide us with a characterization of ordered quasivarieties of pomonoids which are closed under nuclear images. In that case, there is an ordered quasivariety $\class{L}$ of pomonoids such that $\pair{\alg{A}}{\gamma} \in \class{K}$ if and only if $\alg{A} \in \class{L}$. The above theorem then tells us that the nuclear images of pomonoids in $\class{L}$ are axiomatized by the simple quasi-inequations in the signature of nuclear pomonoids valid in $\class{K}$. However, it is not immediately obvious how to relate these to the simple quasi-inequations in the signature of pomonoids valid in~$\class{L}$. We therefore go through the above sequence of above lemmas and theorems and note where they have to be modified to obtain their analogues for pomonoids.

  In the following, let $\class{L}$ be an ordered quasivariety of pomonoids. A quasi-inequation in the signature of pomonoids (s$\ell$-monoids) will be called \emph{simple} if each premise of $\alpha$ has the form $t_{i} \leq x_{i}$ for some variable $x_{i}$.

\begin{lemma}
  Let $\alg{M}$ be a pomonoid and ${\leqL}$ be the smallest $\class{L}$-congruence on the pomonoidal reduct of $\FreeMon{\alg{M}}$. Then $\alg{M}$ is a nuclear image of some pomonoid in $\class{L}$ if and only if $v \leqL w$ in $\FreeMon{\alg{M}}$ implies $\gamma(v) \leq \gamma(w)$ in $\alg{M}$.
\end{lemma}

\begin{proof}
  If $v \leqL w$ in $\FreeMon{\alg{M}}$ implies $\gamma(v) \leq \gamma(w)$ in $\alg{M}$, then ${\leqL}$ is in fact an order congruence also with respect to the nucleus $[\gamma]$. Lemma~\ref{lemma: restriction on free} thus applies, taking $\class{K} \assign \set{\pair{\alg{N}}{\delta}}{\alg{N} \in \class{L}}$. For the opposite direction, apply Lemma~\ref{lemma: restriction on free} taking into account that the smallest $\class{L}$-congruence on the pomonoidal reduct of $\FreeMon{\alg{M}}$ lies below the smallest $\class{K}$-congruence on $\FreeMon{\alg{M}}$.
\end{proof}

  The proofs of the following lemmas and theorems now carry over word for word from their analogues for nuclear pomonoids, keeping in mind that by simple quasi-inequations we now mean simple quasi-inequations in the signature of pomonoids.

\begin{lemma}
  Let $\alg{M}$ be a pomonoid and let $\leqL$ be the smallest $\class{L}$-congruence on the pomonoidal reduct of $\FreeMon{\alg{M}}$. Then $\leqL$ is witnessed by simple quasi-inequations with respect to interpretations over singleton words.
\end{lemma}

\begin{lemma}
  Let $\qine$ be a simple quasi-inequation and $g$ be an interpretation on~$\alg{M}$. If $\qine$ witnesses that $v \leqL w$ in the pomonoidal reduct of $\FreeMon{\alg{M}}$ with respect to the interpretation $h(x) \assign [g(x)]$, then $\qine$ also witnesses that $\gamma(v) \leq \gamma(w)$ in $\alg{M}$ with respect to $g$.
\end{lemma}

\begin{theorem}
  Let $\class{L}$ be an ordered quasivariety of pomonoids. Then the class of all nuclear images of pomonoids from $\class{L}$ is an ordered quasi\-variety of pomonoids axiomatized by the simple quasi-inequational theory of $\class{L}$.
\end{theorem}

\begin{theorem} \label{thm: nuclear images of pomonoids}
  An ordered quasivariety of pomonoids is closed under nuclear images if and only if it is axiomatized by simple quasi-inequations.
\end{theorem}

\begin{proof}
  If $\class{L}$ is an ordered quasivariety of pomonoids axiomatized by a set of simple quasi-inequations, then the previous theorem implies that $\class{L}_{\gamma} = \class{L}$. Conversely, let $\class{L}$ be an ordered quasivariety of pomonoids not axiomatized by any set of simple quasi-inequations. Then there is a pomonoid $\alg{M} \notin \class{K}$ which satisfies the simple quasi-inequational theory of $\class{K}$. By the previous theorem, $\alg{M}$ is a nuclear image of some algebra in $\class{L}$, therefore $\class{L}$ is not closed under nuclear images.
\end{proof}

  Entirely analogous theorems hold for s$\ell$-monoids. Observe that each simple s$\ell$-monoidal quasi-inequation is equivalent (over the class of all s$\ell$-monoids) to a set of simple pomonoidal quasi-inequations: the conclusion $t \leq u$ may be replaced by $u \leq x \implies t \leq x$, where $x$ is a fresh variable, and joins may be eliminated from the left-hand side of each constituent inequality.

\begin{theorem}
  Let $\class{L}$ be a quasivariety of s$\ell$-monoids. Then the class of all nuclear images of s$\ell$-monoids from $\class{L}$ is a quasi\-variety of s$\ell$-monoids axiomatized by the simple quasi-inequational theory of $\class{L}$.
\end{theorem}

\begin{theorem} \label{thm: nuclear images of sl-monoids}
  A quasivariety of s$\ell$-monoids is closed under nuclear images if and only if it is axiomatized by simple quasi-inequations.
\end{theorem}

  We may also consider what happens when we restrict to unital nuclei. These preserve more than just the simple quasi-inequations. Let us call a quasi-inequation in the signature of pomonoids (s$\ell$-monoids) \emph{unital--simple} if each premise has either the form $t_{i} \leq x_{i}$ for some variable $x_{i}$ or the form $t_{i} \leq \1$. An analogous sequence of lemmas and theorems may be stated, where we replace nuclei by unital nuclei, simple quasi-inequations by unital--simple quasi-inequations, and the free nuclear preimage $\FreeMon{\alg{M}}$ by the free unital nuclear preimage $\FreeUMon{\alg{M}}$.

\begin{theorem}
  An ordered quasivariety of pomonoids is closed under unital nuclear images if and only if it is axiomatized by unital--simple quasi-inequations.
\end{theorem}

\begin{theorem}
  A quasivariety of s$\ell$-monoids is closed under unital nuclear images if and only if it is axiomatized by unital--simple quasi-inequations.
\end{theorem}

\section{Residuals and order in free nuclear preimages}

  We now return to our main concern, which is to describe the nuclear images of cancellative residuated lattices. The hope is that, given suitable input, the free semilattice-ordered nuclear preimage will turn out to be a cancellative residuated lattice. Our first order of business is to determine under what conditions the free semilattice-ordered nuclear preimage is lattice-ordered and residuated. These two conditions turn out to be connected: to prove that it is residuated, we will need to use the fact that the order is a lattice.

  Let us start with the residuation properties of free partially ordered nuclear preimages. Unfortunately, none of the free partially ordered nuclear preimage constructions preserve full residuation. Nonetheless, some residuals are preserved. We recall that $\alg{M}$ is a pomonoid and $\gamma$ is a nucleus on $\alg{M}$ in the following.

\begin{fact}
  If the residual $\gamma(u) \bs a$ exists in $\alg{M}$, then the residual $u \bs [a] = [\gamma(u) \bs a]$ exists in $\FreeMon{\alg{M}}$. If $a / \gamma(u)$ exists in $\alg{M}$, then $[a] / u = [a / \gamma(u)]$ exists in $\FreeMon{\alg{M}}$. The same holds for $\FreeUMon{\alg{M}}$, $\FreeSem{\alg{S}}$, and their commutative variants.
\end{fact}

\begin{proof}
  Recall that $w \sqleq [a]$ if and only if $\gamma(w) \leq a$, therefore
\begin{align*}
  u \circ v \sqleq [a] & \iff \gamma(u \circ v) \leq a \\ & \iff \gamma(u) \cdot \gamma(v) \leq a \\ & \iff \gamma(v) \leq \gamma(u) \bs a \\ & \iff v \sqleq [\gamma(u) \bs a]. \qedhere
\end{align*}
\end{proof}

  While full residuation is too much to hope for, it turns out that the free (commutative) unital nuclear preimage construction preserves a weaker property which we call ideal residuation. While residuation requires that the solutions $x$ of the inequality $a \cdot x \leq b$ form a principal downset, ideal residuation only requires that they form a non-empty finitely generated downset.

  The multiplication operation of a pomonoid $\alg{M}$ is \emph{ideally residuated} if there are two binary operations $\bsps\colon \alg{M} \times \alg{M} \to \Id \alg{M}$ and $\sps\colon \alg{M} \times \alg{M} \to \Id \alg{M}$, called the \emph{ideal residuals} of multiplication, such that for all $a, b, c \in \alg{M}$ 
\begin{align*}
  \below a \subseteq c \sps b & \iff a \cdot b \leq c \iff \below b \subseteq a \bsps c.
\end{align*}
  In other words, $\alg{M}$ is ideally residuated if for each $a, c \in \alg{M}$ the downsets
\begin{align*}
  & \set{x \in \alg{M}}{a \cdot x \leq c} & & \text{and} & & \set{x \in \alg{M}}{x \cdot a \leq c}
\end{align*}
  are non-empty and finitely generated.\footnote{Borrowing the terminology used in the theory of equational unification, we might also say that \mbox{residuated} pomonoids have unary residuation type while ideally residuated pomonoids have finitary residuation type, meaning in the latter case that the inequalities $a \cdot x \leq b$ and $x \cdot a \leq b$ have a finite number of maximal elements and each element lies below a maximal one.}

\begin{fact}
  Let us define the following operations on $\Id \alg{M}$:
\begin{align*}
  a \bsps \below \{ x_{1}, \dots, x_{m} \} & \assign \bigcup \set{a \bsps x_{i}}{1 \leq i \leq m}, \\
  \below \{ x_{1}, \dots, x_{m} \} \sps a & \assign \bigcup \set{x_{i} \sps a}{1 \leq i \leq m}.
\end{align*}
  Then in $\Id \alg{M}$
\begin{align*}
  b \bsps (a \bsps c) & = (a \cdot b) \bsps c, & (c \sps b) \sps a & = c \sps (a \cdot b).
\end{align*}
  If $\alg{M}$ has a multiplicative unit $\1$, then $\below \1 \bsps \below a = \below a = \below a \sps \below \1$ in $\Id \alg{M}$.
\end{fact}

  A pomonoid may fail to be residuated for several reasons. The downset $R \assign \set{x \in \alg{M}}{a \cdot x \leq c}$ may be empty for some $a, c \in \alg{M}$. If it is non-empty, $R$ may contain an element not lying below any maximal element of~$R$. If $R$ is generated as a downset by its maximal elements, there may still be infinitely many of them. Finally, $R$ may have finitely many maximal elements but no largest element. All things considered, ideally residuated pomonoids are therefore very close to being residuated: the only thing that they are missing are finite joins commuting with multiplication.

\begin{fact}
  Each ideally residuated s$\ell$-monoid is in fact residuated.
\end{fact}

\begin{proof}
  Let $a \bsps b = \below \{ r_1, \dots, r_n \}$. Then $a \cdot r_i \leq b$ for each $r_i$, so $a \cdot r \leq b$ for $r \assign \bigvee_{i} r_{i}$. Thus $x \leq r$ implies $a \cdot x \leq b$. Conversely, if $a \cdot x \leq b$, then $x \leq r_i \leq r$ for some $r_i$.
\end{proof}

  Free unital nuclear preimages of pomonoids inherit ideal residuation from~$\alg{M}$. In~order to~prove this, it will suffice to find ideal residuals of the forms $[a] \bsps w$ and $w \sps [a]$, since each word is a product of singleton~words.

\begin{fact}
  If $\alg{M}$ is ideally residuated, then so is $\FreeUMon{\alg{M}}$:
\begin{align*}
  [a] \circ u \sqleq [b] \circ v \iff & \text{ either } (u \sqleq [r] \circ v \text{ for some } r \in a \bsps b) \\
  & \text{ or } (\1 \leq b \text{ and } [a] \circ u \sqleq v),
\end{align*}
\begin{align*}
  u \circ [a] \sqleq v \circ [b] \iff & \text{ either } (u \sqleq v \circ [r] \text{ for some } r \in b \sps a) \\
  & \text{ or } (\1 \leq b \text{ and } u \circ [a] \sqleq v).
\end{align*}
  If $\alg{M}$ is an integral residuated pomonoid, then so is $\FreeUMon{\alg{M}}$:
\begin{align*}
  [a] \circ u \sqleq [b] \circ v \iff & u \sqleq [a \bs b] \circ v, \\
  u \circ [a] \sqleq v \circ [b] \iff & u \sqleq v \circ [b / a].
\end{align*}
  If $\alg{M}$ is commutative and ideally residuated, then so is $\CFreeRedUMon{\alg{M}}$:
\begin{align*}
  [a] \circ u \sqleq v \iff & \text{ there is some } b \in \alg{M} \text{ and some } w \in \CFreeRedUMon{\alg{M}} \text{ such that } \\
  & \text{ either } (v = [b] \circ w \text{ and } u \sqleq [r] \circ w \text{ for some } r \in a \bsps b = b \sps a) \\
  & \text{ or } (\1 \leq b \text{ and } u \circ [a] \sqleq w).
\end{align*}
\end{fact}

\begin{proof}
  First observe that repeatedly applying these equivalences will give us a non-empty finite set of attainable upper bounds for $u$, therefore they suffice to establish ideal residuation.

  If $[a] \circ u \sqleq [b] \circ v$, then either $\emptyword \sqleq [b] \circ v_{1}$ and $[a] \circ u \sqleq v_{2}$ for some $v_{1} \circ v_{2} = v$ or $[a] \circ u_{1} \sqleq [b]$ and $u_{2} \sqleq v$ for some $u_{1} \circ u_{2} = u$. In the former case $\emptyword \sqleq [b]$, so $\1 \leq b$ and $[a] \circ u \sqleq v$. Conversely, if $\1 \leq b$ and $[a] \circ u \sqleq v$, then $\emptyword \sqleq [b]$ and $[a] \circ u \sqleq [b] \circ v$. In the latter case, $u_{1} \sqleq [r]$ for some $r \in a \bsps b$, so $u = u_{1} \circ u_{2} \sqleq r \circ v$. Conversely, if $u \sqleq r \circ v$ for some $r \in a \bsps b$, then $[a] \circ u \sqleq [a] \circ [r] \circ v \sqleq [a \cdot r] \circ v \sqleq [b] \circ v$. The proof in the commutative case is entirely analogous.

  For integral pomonoids, each word in $\FreeUMon{\alg{M}}$ is either equivalent to $[\1]$ (and $[a] \circ u \sqleq [\1]$ for each word $u$) or it is equivalent to a word of the form $[b] \circ v$ for $b < \1$. In the latter case $[a] \circ u \sqleq [b] \circ v$ if and only if $u \sqleq [a \bs b] \circ v$. Similarly for the other residual.
\end{proof}

  The pomonoids $\FreeMon{\alg{M}}$ and $\CFreeRedMon{\alg{M}}$ are never ideally residuated: there is no word $v$ such that $u \circ v \sqleq \emptyword$ or $v \circ u \sqleq \emptyword$ if $u$ is non-empty.

  One difference between the posemigroup case and the pomonoidal case is that $\FreeSem{\alg{S}}$ does in fact inherit full residuation from $\alg{S}$. Unfortunately, this does not extend to the commutative case.

\begin{fact}
  If $\alg{S}$ is ideally residuated, then so is $\FreeSem{\alg{S}}$:
\begin{align*}
  [a] \circ u \sqleq [b] \circ v \iff & u \sqleq [r] \circ v \text{ for some } r \in a \bsps b, \\
  u \circ [a] \sqleq v \circ [b] \iff & u \sqleq v \circ [r] \text{ for some } r \in b \sps a.
\end{align*}
  If $\alg{S}$ is residuated, then so is $\FreeSem{\alg{S}}$:
\begin{align*}
  [a] \circ u \sqleq [b] \circ v \iff & u \sqleq [a \bs b] \circ v, \\
  u \circ [a] \sqleq v \circ [b] \iff & u \sqleq v \circ [b / a].
\end{align*}
  If $\alg{S}$ is commutative and ideally residuated, then so is $\CFreeSem{\alg{S}}$:
\begin{align*}
  [a] \circ u \sqleq v \iff & \text{ there is some } b \in \alg{S} \text{ and some } w \in \CFreeSem{\alg{S}} \text{ such that } \\
  & ~ v = [b] \circ w \text{ and } u \sqleq [r] \circ w \text{ for some } r \in a \bsps b = b \sps a.
\end{align*}
\end{fact}

  Starting with a (commutative) residuated s$\ell$-monoid $\alg{M}$, in general $\FreeUMon{\alg{M}}$ ($\CFreeRedUMon{\alg{M}}$) will thus only be ideally residuated. We now show that under favorable order-theoretic conditions, full residuation is recovered at the level of $\Id \FreeUMon{\alg{M}}$ ($\Id \CFreeRedUMon{\alg{M}}$).

\begin{lemma}
  Let $\alg{M}$ be an ideally residuated pomonoid. Then the two residuals $\below x \bs \below \{ z_{1}, \dots, z_{n} \}$ and $\below \{ z_{1}, \dots, z_{n} \} / \below x$ exist in $\Id \alg{M}$ for $x, z_{1}, \dots, z_{n} \in \alg{M}$.
\end{lemma}

\begin{proof}
  Let $c \assign \below \{ z_{1}, \dots, z_{p} \}$ Then $\below x \cdot b \leq c$ in $\Id \alg{M}$ for $b = \below \{ y_1, \dots, y_{n} \} = \below y_{1} \vee \dots \vee \below y_n$ if and only if $\below x \cdot \below y_j \leq c$ in $\Id \alg{M}$ for each~$y_j$, that is, if and only if for each $y_j$ there is some $z_k$ such that $x \cdot y_j \leq z_k$ in~$\alg{M}$. By ideal residuation, this is equivalent to $y_j \in x \bsps z_{k}$. That is, for each $y_j$ there is a $z_k$ such that $y_j \leq r_j$ for some $r_j \in x \bsps z_{k}$. Thus $\below x \cdot b \leq c$ if and only if there is a function $f\colon \{1, \dots, n \} \to B$ such that $y_j \leq f(j)$ in $\alg{M}$ for each $y_j$, where $B \assign \bigcup \set{x \bsps z_{k}}{1 \leq k \leq p}$.

  For each such function $f$ there is a largest tuple $y_1, \dots, y_n$ satisfying these constraints, namely the elements $f(j)$ themselves. Thus $x \cdot b \leq c$ in $\Id \alg{M}$ if and only if there is a function $f\colon \{ 1, \dots, n \} \to B$ such that $b \leq \below f(1) \vee \dots \vee \below f(n)$. This is equivalent to $b \leq \below \{ u_1, \dots, u_n \}$ for some finite subset $\{ u_1, \dots, u_n \}$ of $B$. But there is a largest finite subset of $B$, namely $B$ itself, so this is equivalent to $b \leq \below B$. (Note that the set $B$ is given by our choice of $\below x$ and $c$, while $b$ and thus also~$n$ are variables in this argument.) An analogous argument proves the existence of the residual~$c / a$.
\end{proof}

\begin{lemma}
  If $\alg{M}$ is ideally residuated and $\Id \alg{M}$ a lattice, then $\Id \alg{M}$ is a residuated lattice.
\end{lemma}

\begin{proof}
  It suffices to observe that $\below \{ x_{1}, \dots, x_{m} \} = \bigvee_{1 \leq i \leq m} \below x_{i} \vee \dots \vee \below x_{m}$, so $\below \{ x_{1}, \dots, x_{m} \} \bs \below \{ y_{1}, \dots, y_{n} \} = \bigwedge_{1 \leq i \leq m} (\below x_{i} \bs \below \{ y_{1}, \dots, y_{n} \})$.
\end{proof}

  In general, $\Id P$, for an arbitrary poset $P$, might fail to be a lattice for two separate reasons. Firstly, there may be two principal downsets whose intersection is empty. This occurs if and only if $P$ fails to be down-directed. Secondly, there may be two principal downsets whose intersection is not finitely generated. The simplest way to avoid this is to assume that each downset of $P$ is finitely generated, which is one of several equivalent definitions of a dually well-partially-ordered poset. A more common definition is that a poset is a \emph{dual well-partial-order (dual wpo)} if it has no infinite ascending chains and no infinite antichains.

\begin{lemma}
  Let $\alg{M}$ be a (commutative) down-directed pomonoid. Then $\FreeUMon{\alg{M}}$ ($\CFreeRedUMon{\alg{M}}$) is down-directed.
\end{lemma}

\begin{proof}
  Given any two non-empty words $u = [a_1, \dots, a_m]$ and $v = [b_1, \dots, b_n]$, where without loss of generality $m \leq n$, there is a word $w = [c_{1}, \dots, c_{n}]$ such that $w \sqleq u$. For example, we may pad $u$ by $[\1, \dots, \1]$ on the left. By down-directedness there are $d_{i}$ such that $d_{i} \leq b_{i}$ and $d_{i} \leq c_{i}$, so $[d_{1}, \dots, d_{n}]$ is a lower bound of $u$ and $v$.
\end{proof}

\begin{lemma}
  Let $\alg{S}$ be a (commutative) down-directed posemigroup such that for each $c \in \alg{S}$ there are $a, b \in \alg{S}$ with $a \cdot b \leq c$. Then $\FreeSem{\alg{S}}$ ($\CFreeSem{\alg{S}}$) is down-directed.
\end{lemma}

\begin{proof}
  The additional assumption ensures that for each $u = [a_1, \dots, a_m]$ and $m \leq n$ there is some word $w \assign [a_{b}, \dots, a_{n}] \sqleq u$. For example, we may take $w$ to be $[c_{1}, \dots, c_{k}, a_{2}, \dots, a_{m}]$ where $k = n - m + 1$ and $c_{1} \cdot \ldots \cdot c_{k} \leq a_{1}$.
\end{proof}

  In particular, this holds for each (ideally) residuated down-directed po\-semi\-group, and for each posemigroup with a bottom element such that $\bot \cdot \bot = \bot$.
 
  The pomonoids $\FreeMon{\alg{M}}$ and $\CFreeRedMon{\alg{M}}$ are never down-directed: the empty word is a minimal element but not the smallest element (unless $\alg{M}$ is trivial).

  As with residuation, in general the property of being a dual wpo is not fully inherited by free partially ordered nuclear preimages beyond the integral case. If an integral (commutative) pomonoid $\alg{M}$ is a dual wpo, then $\FreeMon{\alg{M}}$ and $\FreeUMon{\alg{M}}$ ($\CFreeRedMon{\alg{M}}$ and $\CFreeRedUMon{\alg{M}}$) are also dual wpos, by the observation of Galatos and Hor\v{c}\'{i}k~\mbox{\cite[Lemma~4.2]{galatos+horcik13}} that an integral pomonoid generated by a dually wpo subset is a dual wpo. Beyond the integral case, however, the free partially ordered nuclear preimage of a finite pomonoid might fail to be a dual wpo. If $\1 < a < a^2$ in $\alg{M}$, then $[a] < [a,a] < [a,a,a] < \dots$ is an infinite ascending chain in $\FreeMon{\alg{M}}$. Similarly, if $\1, a, a^2 = a^3$ is a three-element antichain in $\alg{M}$, then $\{ [a], [a,a], [a,a,a], \dots \}$ is an infinite antichain in $\FreeMon{\alg{M}}$. We therefore need to consider a weaker property.

\begin{lemma}
  Let $\alg{M}$ be a (commutative) dually wpo pomonoid. Then for each $n \in \omega$ the restriction of $\FreeMon{\alg{M}}$ and $\FreeUMon{\alg{M}}$ ($\CFreeRedMon{\alg{M}}$ and $\CFreeRedUMon{\alg{M}}$) to words of length at most $n$ is a dual wpo.
\end{lemma}

\begin{proof}
  A crucial observation is that when restricted to words of length precisely~$n$, the preorder on $\FreeMon{\alg{M}}$ extends the componentwise order on $\alg{M}^{n}$. Because the product of finitely many dual wpos is known to be a dual wpo, this componentwise order is a dual wpo. In particular, there is no infinite antichain of words of length precisely $n$ in $\alg{M}^{n}$, and therefore there is also no infinite antichain of words of length precisely $n$ in~$\FreeMon{\alg{M}}$. It follows that for each $n$ there is no infinite antichain in $\FreeMon{\alg{M}}$ of words of length at most~$n$.

  It remains to prove, by induction over $n$, that there is no infinite ascending chain $u_{1} < u_{2} < \dots$ of words of length at most $n$ in $\FreeMon{\alg{M}}$. If $n = 1$, this is impossible because $\alg{M}$ is dually wpo. Now suppose that this claim holds for some $n-1$ and there is an infinite ascending chain $u_{1} < u_{2} < \dots$ of words of length at most $n$. If for infinitely many $u_{i}$ there are words $v_{i}$ of length less than~$n$ with $u_{i} \sqleq v_{i} \sqleq u_{i+1}$, the claim follows from the inductive hypothesis, since the elements $v_{i}$ would form at infinite ascending chain of words of length less than $n$. We may therefore assume that for each $u_{i}$ there is no word $v$ of length less than $n$ with $u_{i} \sqleq v \sqleq u_{i+1}$. In particular, each of the words $u_{i}$ has length $n$. But this implies that the decomposition of $u_{i}$ witnessing that $u_{i} \sqleq u_{i+1}$ contains no instances of the empty word. (Otherwise, there are decompositions $u_{i} = u'_{i} \circ \emptyword \circ u''_{i}$ and $u_{i+1} = u'_{i+1} \circ [a] \circ u''_{i+1}$ such that $u'_{i} \sqleq u'_{i+1}$, $\emptyword \sqleq [a]$, $u''_{i} \sqleq u''_{i+1}$. But then we have a word $v_{i} = u'_{i+1} \circ u''_{i+1}$ of length $n-1$ with $u_{i} \sqleq v_{i} \sqleq u_{i+1}$, contradicting our assumption.) Because the decomposition contains no instances of the empty word and $u_{i}$ and $u_{i+1}$ have the same length, this implies that for some $a_{1}, \dots, a_{n}, b_{1}, \dots, b_{n} \in \alg{M}$ we have $u_{i} = [a_{1}, \dots, a_{n}]$ and $u_{i+1} = [b_{1}, \dots, b_{n}]$ with $a_{i} \leq b_{i}$ for $1 \leq i \leq n+1$. But this is precisely the componentwise order on $\alg{M}^{n}$, which is a dual wpo, therefore no such infinite ascending chain exists.
\end{proof}

\begin{lemma}
  Let $\alg{S}$ be a (commutative) dually wpo posemigroup. Then for each $n \in \omega$ restricting $\FreeSem{\alg{S}}$ ($\CFreeSem{\alg{S}}$) to words of length precisely $n$ yields a dual wpo.
\end{lemma}

\begin{fact}
  Let $\alg{M}$ be a (commutative) down-directed dually wpo po\-monoid. Then $\Id \FreeUMon{\alg{M}}$ and $\Id \FreeMon{\alg{M}}$ ($\Id \CFreeRedUMon{\alg{M}}$ and $\Id \CFreeRedMon{\alg{M}}$) are lattices.
\end{fact}

\begin{proof}
  If $w \sqleq u \assign [a_1, \dots, a_m]$ and $w \sqleq v \assign [b_1, \dots, b_n]$, then there is a decomposition $w = u_1 \circ \ldots \circ u_{k}$ with $k \leq m \cdot n$ (in fact, with $k \leq m+n-1$ in the non-commutative case) fine enough to witness both of these inequalities in the sense that there is some $i_1 \leq i_2 \leq \dots \leq i_{m}$ and $j_1 \leq j_2 \leq \dots \leq j_n$ such that $u_1 \circ \ldots u_{i_{1}} \sqleq a_1$, $u_{i_{1}+1} \circ \ldots \circ u_{i_{2}} \sqleq a_2$, etc.\ and likewise $u_1 \circ \ldots u_{j_{1}} \sqleq b_1$, $u_{j_{1}+1} \circ \ldots \circ u_{j_{2}} \sqleq b_2$, etc. Taking $w' \assign [\gamma(u_1), \dots, \gamma(u_k)]$, we have $w \sqleq w' \sqleq u, v$. That is, each lower bound of both $u$ and $v$ lies below a lower bound of length at most $m \cdot n$ (in fact, at most $m+n-1$ in the non-commutative case). The fact that each such lower bound is below one of finitely many maximal lower bounds now follows from the fact that the restriction of the free nuclear preimage to words of length at most $m \times n$ is a sectional dual wpo. (If $u$ or $v$ is the empty word, then $w = \emptyword$.)
\end{proof}

\begin{fact}
  Let $\alg{S}$ be a (commutative) down-directed dually wpo posemigroup where for each $c \in \alg{S}$ there are $a, b \in \alg{S}$ with $a \cdot b \leq c$. Then $\Id \FreeSem{\alg{S}}$ ($\Id \CFreeSem{\alg{S}}$) is a lattice.
\end{fact}

  Summing up, applying the free semilattice-ordered nuclear preimage to a finite, or more generally dually wpo, residuated lattice yields a residuated lattice.

\begin{theorem}
  If $\alg{M}$ is a dually wpo (commutative) [integral] residuated lattice, then $\Id \FreeMon{\alg{M}}$ ($\Id \CFreeRedMon{\alg{M}}$) is a (commutative) [integral] residuated lattice.
\end{theorem}

  Observe that a dually wpo residuated s$\ell$-monoid is a residuated lattice provided that it is down-directed and the inequalities $a \cdot x \leq b$ and $y \cdot a \leq b$ have at least one solution for each $a, b$. No substantial generality is thus to be gained in the above theorem by extending it to from residuated lattices $\alg{M}$ to s$\ell$-monoids.

\section{Nuclear images of cancellative pomonoids}

  Having described several variants of the free nuclear preimage construction and shown that under suitable assumptions they yield structures of the appropriate types, from pomonoids to residuated lattices, in the rest of the paper we deal with the problem of ensuring that the resulting structures are cancellative. Naturally, we start with the partially ordered case.

\begin{definition}
  A pomonoid is \emph{cancellative} if it satisfies
\begin{align*}
  a \cdot x \leq b \cdot x & \implies a \leq b, &
  x \cdot a \leq x \cdot b & \implies a \leq b.
\end{align*}
\end{definition}

  More explicitly, this notion should perhaps be called order cancellativity. Ordinary cancellativity is in fact usually defined by the weaker implications
\begin{align*}
  a \cdot x = b \cdot x & \implies a = b, & x \cdot a = x \cdot b & \implies a = b.
\end{align*}
  However, we shall have no use for this weaker notion of cancellativity, which is less appropriate in the study of pomonoids. The two definitions coincide for s$\ell$-monoids: if $a \cdot x \leq b \cdot x$, then $(a \vee b) \cdot x = (a \cdot x) \vee (b \cdot x) = b \cdot x$, so $a \vee b = b$.

  The free nuclear preimage $\FreeMon{\alg{M}}$ is never cancellative: $\emptyword \circ [\1] = [\1] = [\1] \circ [\1]$ in $\FreeMon{\alg{M}}$, but $\emptyword$ and $[\1]$ are distinct words. The same holds for $\CFreeRedMon{\alg{M}}$ if $\alg{M}$ is commutative. The appropriate question to ask is therefore under what conditions $\FreeUMon{\alg{M}}$ and $\CFreeRedUMon{\alg{M}}$ are cancellative. We follow the terminology of~\cite{gil-ferez+lauridsen+metcalfe20}.

\begin{definition}
  A posemigroup is \emph{integrally closed} if it satisfies
\begin{align*}
  a \cdot x \leq a & \implies x \cdot b \leq b, &
  x \cdot a \leq a & \implies b \cdot x \leq b.
\end{align*}
\end{definition}

  Observe that these are simple quasi-inequations in the sense of the previous section. For pomonoids they simplify to
\begin{align*}
  a \cdot x \leq a & \implies x \leq \1, &
  x \cdot a \leq a & \implies x \leq \1.
\end{align*}
  For residuated posemigroups these implications are equivalent to the equation $x \bs x = y / y$. In residuated pomonoids we have $x \bs x = \1 = x / x$.

\begin{proposition}
  Let $\alg{M}$ be a (commutative) pomonoid. Then the pomonoid $\FreeUMon{\alg{M}}$ ($\CFreeRedUMon{\alg{M}}$) is cancellative if and only if $\alg{M}$ is integrally closed.
\end{proposition}

\begin{proof}
  If $a \cdot x \leq a$ but $x \nleq \1$, then $[a] \circ [x] \sqleq [a] = [a] \circ [\1]$ but $[x] \nsqleq [\1]$, so cancellativity fails. Conversely, suppose that the first implication holds in~$\alg{M}$. It suffices to prove that $[a] \circ u \sqleq [a] \circ v$ implies $u \sqleq v$ for each $a \in \alg{M}$.

  We prove this claim by induction over the length of $v$. If~$v = \emptyword$, then $[a] \circ u \sqleq [a]$, so $a \cdot \gamma(u) \leq a$. The assumed implication yields $\gamma(u) \leq \1$, so indeed $u \sqleq [\1] \sqleq \emptyword = v$.

  Now suppose that $v = [b] \circ v'$, i.e.\ $[a] \circ u \sqleq [a] \circ [b] \circ v'$. There are several cases to consider depending on how this inequality is witnessed, in particular on which part of the relevant decomposition of $[a] \circ u$ lies below $[a]$.

  If $[a] \circ u_{1} \sqleq [a]$ and $u_{2} \sqleq [b] \circ v'$ for some decomposition $u = u_{1} \circ u_{2}$, then as before $u_{1} \sqleq \emptyword$, so $u = u_{1} \circ u_{2} \sqleq u_{2} \sqleq [b] \circ v'$. If on the other hand $\emptyword \sqleq [a] \circ [b] \circ v'_{1}$ and $[a] \circ u \sqleq v'_{2}$ for some decomposition $v' = v'_{1} \circ v'_{2}$, then $\emptyword \sqleq [a]$, $\emptyword \sqleq [b]$, and $\emptyword \sqleq v'_{1}$, so $u = \emptyword \circ u \sqleq [a] \circ u \sqleq v'_{2} \sqleq [b] \circ v'_{1} \circ v'_{2} = [b] \circ v'$. If $\emptyword \sqleq [a]$ and $[a] \circ u \sqleq [b] \circ v'$, then $u = \emptyword \circ u \sqleq [a] \circ u \sqleq [b] \circ v'$.

  In the commutative case, we only have to discuss one more way of witnessing that $[a] \circ u \sqleq [a] \circ [b] \circ v'$. If $u_{1} \sqleq [a]$ and $[a] \circ u_{2} \sqleq [b] \circ v'$ for some decomposition $u = u_{1} \circ u_{2}$, then $u = u_{1} \circ u_{2} \sqleq [a] \circ u_{2} \sqleq [b] \circ v'$.
\end{proof}

\begin{proposition}
  Let $\alg{S}$ be a (com\-mutative) posemigroup. Then the posemigroup $\FreeSem{\alg{S}}$ ($\CFreeSem{\alg{S}}$) is cancellative if and only if $\alg{S}$ is integrally closed.
\end{proposition}

\begin{proof}
  If $a \cdot x \leq a$ but $x \cdot b \nleq b$, then $[a] \circ [x,b] = [a,x,b] \leq [a,b] = [a] \circ [b]$ but $[x,b] \nleq [b]$, so cancellativity fails. Conversely, suppose that the first implication holds in~$\alg{S}$. It suffices to prove that $[a] \circ u \sqleq [a] \circ v$ implies $u \sqleq v$ for each $a \in \alg{S}$. If~$v = [b]$, then $[a] \circ u \sqleq [a] \circ [b]$ implies that either we directly obtain a decomposition witnessing that $u \sqleq [b] = v$ or there is a decomposition $u = u_{1} \circ u_{2}$ such that $[a] \circ u_{1} \sqleq [a]$ and $u_{2} \sqleq [b]$, hence $a \cdot \gamma(u_{1}) \leq a$ and $\gamma(u_{2}) \leq b$. The assumed implication now yields that $\gamma(u_{1}) \cdot \gamma(u_{2}) \leq \gamma(u_{2}) \leq b$ and $u \sqleq [b] = v$.

  If $v = [b] \circ v'$, then either we directly obtain a decomposition witnessing that $u \sqleq [b] \circ v' = v$ or there is a decomposition $u = u_{1} \circ u_{2} \circ u_{3}$ such that $[a] \circ u_{1} \sqleq [a]$ and $u_{2} \sqleq [b]$ and $u_{3} \sqleq v'$. Then $[a] \circ u_{1} \circ u_{2} \sqleq [a] \circ b$, hence by the above $u_{1} \circ u_{2} \sqleq [b]$ and $u = u_{1} \circ u_{2} \circ u_{3} \sqleq [b] \circ v' = v$.

  In the commutative case, we only have to discuss one more way of witnessing that $[a] \circ u \sqleq [a] \circ [b] \circ v'$. If $u_{1} \sqleq [a]$ and $[a] \circ u_{2} \sqleq [b] \circ v'$ for some decomposition $u = u_{1} \circ u_{2}$, then $u = u_{1} \circ u_{2} \sqleq [a] \circ u_{2} \sqleq [b] \circ v'$.
\end{proof}

  The following theorems, which hold both in the commutative and the non-commutative case, now immediately follow. Here we call a posemigroup \emph{integral} if it satisfies the inequalities $x \cdot y \leq x$ and $x \cdot y \leq y$.

  Observe that in all the theorems in the present section describing the nuclear images of certain pomonoids, we may moreover require that the nuclear images in question are upward closed, i.e.\ that the closed elements form an upset.

\begin{theorem} \label{thm: integrally closed pomonoids}
  Integrally closed (integral) [commutative] pomonoids are precisely the nuclear images of (integral) [commutative] cancellative pomonoids.
\end{theorem}

\begin{theorem}
  Integrally closed (integral) [commutative] posemigroup are precisely the nuclear images of (integral) [commutative] cancellative posemigroups.
\end{theorem}

  This next batch of theorems follows if we take into account that free nuclear preimages always preserve ideal residuation, and that moreover in the integral non-commutative case and the posemigroup case they preserve residuation.

\begin{theorem} \label{thm: integral residuated pomonoids}
  Integrally closed (integral) [commutative] ideally residuated po\-monoids are precisely the nuclear images of (integral) [commutative] cancellative ideally residuated pomonoids. Integral residuated pomonoids are precisely the nuclear images of integral cancellative residuated pomonoids.
\end{theorem}

\begin{theorem}
  Integrally closed (integral) residuated posemigroups are precisely the nuclear images of (integral) cancellative residuated posemigroups.
\end{theorem}

  Every commutative cancellative pomonoid embeds into an Abelian pogroup, namely its pogroup of fractions. We therefore immediately obtain a link with pogroups in the commutative case.

\begin{theorem}
  Integrally closed (integral) commutative pomonoids are precisely the nuclear images of subpomonoids of (negative cones of) Abelian pogroups.
\end{theorem}

  Beyond the commutative case, determining whether a cancellative pomonoid embeds into a pogroup may be quite difficult. In our particular case, however, it turns out that if $\FreeUMon{\alg{M}}$ is cancellative, it always embeds into a pogroup.

\begin{table} 
\caption{Rules which define the preorder $\sqleq$ on $\FreeGroup{\alg{M}}$}
\label{tab: rules}
\begin{framed}
\textbf{Positive monotonicity}
\begin{align*}
  u & \sqleq [a] \qquad \text{if $u \sqleq [a]$ in $\FreeUMon{\alg{M}}$}
\end{align*}
\textbf{Negative monotonicity}
\begin{align*}
  [a]^{-1} & \sqleq u^{-1} \qquad \text{if $u \sqleq [a]$ in $\FreeUMon{\alg{M}}$}
\end{align*}
\textbf{Contraction}
\begin{align*}
  u^{-1} v w^{-1} & \sqleq x \qquad \text{if $v \sqleq u \circ x \circ w$ in $\FreeUMon{\alg{M}}$ (for $u$, $w$ not both empty)}
\end{align*}
\textbf{Expansion}
\begin{align*}
   x & \sqleq u^{-1} v w^{-1} \qquad \text{if $u \circ x \circ w \sqleq v$ in $\FreeUMon{\alg{M}}$ (for $u$, $w$ not both empty)}
\end{align*}
\textbf{Permutation}
\begin{align*}
  u v^{-1} & \sqleq x^{-1} y \qquad \text{if $x \circ u \sqleq y \circ v$ in $\FreeUMon{\alg{M}}$ (for $v$, $x$ both non-empty)} \\
  u^{-1} v & \sqleq x y^{-1} \qquad \text{if $v \circ y \sqleq u \circ x$ in $\FreeUMon{\alg{M}}$ (for $u$, $y$ both non-empty)}
\end{align*}
\end{framed}
\end{table}

  Let $\FreeGroup{\alg{M}}$ denote the free monoid generated by elements of the form $[a]$ for $a \in \alg{M}$ (\emph{positive letters}) and elements of the form $[a]^{-1}$ for $a \in \alg{M}$ (\emph{negative letters}), where the exponent is a purely formal symbol. Words in $\FreeGroup{\alg{M}}$ will be denoted $\alpha$, $\beta$, $\gamma$, their products will be denoted either $\alpha \circ \beta$ or simply $\alpha \beta$. Words which only contain positive letters are called \emph{positive words}. If $u = [a_{1}, \dots, a_{n}] = [a_{1}] \circ \ldots \circ [a_{n}]$ is a positive word, we use the notation $u^{-1} \assign [a_{n}]^{-1} \circ \ldots \circ [a_{1}]^{-1}$.

  We now define a preorder $\sqleq$ on $\FreeGroup{\alg{M}}$ as the reflexive transitive closure of all instances of the inequalities shown in Table~\ref{tab: rules}. Here $u$, $v$ etc.\ range over all positive words, possibly empty. These are to be interpreted as applying in any context within a word, e.g.\ the contraction inequality is more explicitly the inequality $\alpha u^{-1} v w^{-1} \beta \sqleq \alpha x \beta$ for all $\alpha, \beta \in \FreeGroup{\alg{M}}$.

  It will be useful to view this as a proof system in which we can prove that $\alpha \sqleq \beta$ by applying certain rules. For example, suppose that $[a_{1}] \circ [a_{2}] \sqleq [a]$, $u \sqleq [b] \circ [a_{2}]$, and $v \sqleq [a_{1}] \circ [c]$ hold in $\FreeUMon{\alg{M}}$. Then the sequence of words
\begin{align*}
  u \circ [a]^{-1} \circ v \sqleq u \circ [a_{2}]^{-1} \circ [a_{1}]^{-1} \circ v \sqleq [b] \circ [a_{1}]^{-1} \circ v \sqleq [b] \circ [c]
\end{align*}
  may be seen as a proof that $u \circ [a]^{-1} \circ v \sqleq [b] \circ [c]$ which uses negative monotonicity and two instances of contraction.

  Clearly $\FreeMon{\alg{M}}$ is a submonoid of $\FreeGroup{\alg{M}}$. Moreover, the embedding of $\FreeUMon{\alg{M}}$ into $\FreeGroup{\alg{M}}$ is order preserving. The difficult part is showing that it is order reflecting.

\begin{fact}
  If $u \sqleq v$ in $\FreeUMon{\alg{M}}$, then $u \sqleq v$ in $\FreeGroup{\alg{M}}$.
\end{fact}

\begin{proof}
  This holds by positive monotonicity and the fact that $[\1] \sqleq [\1] \circ [\1]^{-1} \sqleq \emptyword$.
\end{proof}

  The preorder $\sqleq$ is an order congruence of $\FreeGroup{\alg{M}}$: combining a proof of $\alpha_{1} \sqleq \beta_{1}$ and a proof of $\alpha_{2} \sqleq \beta_{2}$ yields a proof of $\alpha_{1} \circ \alpha_{2} \sqleq \beta_{1} \circ \beta_{2}$. The quotient $\FreeGroup{\alg{M}} / {\sqleq}$ is a pomonoid which we denote $\FreeRedGroup{\alg{M}}$.

\begin{fact}
  $\FreeRedGroup{\alg{M}}$ is a pogroup where the inverse of the equivalence class of $u$ is the equivalence class of $u^{-1}$.
\end{fact}

\begin{proof}
  This holds because $u^{-1} u \sqleq \emptyword$ and $u u^{-1} \sqleq \emptyword$ are instances of contraction and $\emptyword \sqleq u^{-1} u$ and $\emptyword \sqleq u u^{-1}$ are instances of expansion.
\end{proof}

\begin{fact}
  $\FreeRedGroup{\alg{M}}$ is the free pogroup over the pomonoid $\FreeUMon{\alg{M}}$.
\end{fact}

\begin{proof}
  Clearly the free pogroup over the pomonoid $\FreeUMon{\alg{M}}$ (which exists by general category-theoretic considerations) validates all of the inequalities of Table~\ref{tab: rules}. Since conversely these inequalities ensure that the quotient $\FreeRedGroup{\alg{M}}$ is a group, it must be the free pogroup over the pomonoid $\FreeUMon{\alg{M}}$.
\end{proof}

  We now show that each proof in the proof system defined by Table~\ref{tab: rules} may be transformed into a certain normal form. A proof of an inequality $\alpha \sqleq \beta$ is \emph{normal} if it consists of some instances (possibly none) of negative monotonicity, followed by some instances (possibly none) of contraction, followed similarly by permutation, expansion, and positive monotonicity. In other words, disregarding monotonicity, we first apply rules which eliminate (blocks of) inverses, then rules which preserve inverses, and finally rules which introduce inverses.

\begin{theorem}
  Let $\alg{M}$ be an integrally closed pomonoid. Then each inequality which holds in $\FreeGroup{\alg{M}}$ has a normal proof.
\end{theorem}

\begin{proof}
  It suffices to show that two successive instances of rules which are not in the correct relative order may be replaced by a sequence of instances of rules which are in the correct order. For example, contraction, expansion, or permutation followed by negative monotonicity may easily be reduced simply to contraction, expansion, or permutation. Similarly, contraction, expansion, or permutation preceded by positive monotonicity may easily be reduced simply to contraction, expansion, or permutation. Of course, neighbouring instances of positive and negative monotonicity may be permuted.

  This leaves contractions, permutations, and expansions to be dealt with in more detail. We shall not explicitly discuss the side conditions requiring that certain words be non-empty, since these are easy to handle. For example, if $u = w = \emptyword$, then instead of applying the contraction rule $u^{-1} v w^{-1} \sqleq x$ we apply positive monotonicity $v \sqleq x$.

  The proof is a tedious case analysis. Throughout, we suppress the context in which rules are applied. For example, the first subcase below could more explicitly be described as transforming an instance of expansion $\alpha p^{-1} u q^{-1} \beta \sqleq \alpha p^{-1} v^{-1} w x^{-1} q^{-1} \beta$ followed by an instance of contraction $\alpha p^{-1} v^{-1} w x^{-1} q^{-1} \beta \sqleq \alpha r \beta$ into a single instance of contraction, $\alpha p^{-1} u q^{-1} \beta \sqleq \alpha r \beta$.

  Expansion followed by contraction:
\begin{enumerate}[(i)]
\item $u \sqleq v^{-1} w x^{-1}$ and $p^{-1} v^{-1} w x^{-1} q^{-1} \sqleq r$: in that case $v u x \sqleq w \sqleq v p r q x$, so by cancellativity $u \sqleq p r q$ and $p^{-1} u q^{-1} \sqleq r$ (contraction).
\item $u \sqleq v_{2}^{-1} v_{1}^{-1} w x_{2}^{-1} x_{1}^{-1}$ and $v_{1}^{-1} w x_{2}^{-1} \sqleq r$: in that case $v_{1} v_{2} u x_{1} x_{2} \sqleq w \sqleq v_{1} r x_{2}$, so by cancellativity $v_{2} u x_{1} \sqleq r$ and $u \sqleq v_{2}^{-1} r x_{1}^{-1}$ (expansion). Recall that $\FreeUMon{\alg{M}}$ is cancellative because $\alg{M}$ is integrally closed.
\item $u \sqleq v^{-1} w x_{2}^{-1} x_{1}^{-2}$ and $p^{-1} v^{-1} w x_{2}^{-1} \sqleq r$: in that case $v u x_{1} x_{2} \sqleq w \sqleq v p r x_{2}$, so by cancellativity $u x_{1} \sqleq p w$ and $p^{-1} u \sqleq r x_{1}^{-1}$ (permutation).
\item $u \sqleq v_{2}^{-1} v_{1}^{-1} w x^{-1}$ and $v_{1}^{-1} w x^{-1} q^{-1} \sqleq r$: analogous to the previous case.
\item $u \sqleq v^{-1} w x_{2}^{-1} x_{1}^{-1}$ and $x_{1}^{-1} p q^{-1} \sqleq r$: in that case $v u x_{1} x_{2} \sqleq w$ and $p \sqleq x_{1} r q$, so $u p q^{-1} \sqleq u x_{1} r \sqleq v^{-1} w x_{2}^{-1} r$ (contraction and expansion).
\item $u \sqleq v_{2}^{-1} v_{1}^{-1} w x^{-1}$ and $p^{-1} q v_{2}^{-1} \sqleq r$: analogous to the previous case.
\item $u \sqleq w x^{-1}$ and $p^{-1} q w x^{-1} y^{-1} \sqleq r$: in that case $u x \sqleq w$ and $q w \sqleq p r y x$, so $q u x \sqleq q w \sqleq p r y x$. Thus by cancellativity $q u \sqleq p r y$ and $p^{-1} q u y^{-1} \sqleq r y y^{-1} \sqleq r$ (contraction).
\item $u \sqleq v_{2}^{-1} v_{1}^{-1} w x^{-1}$ and $p^{-1} q v_{2}^{-1} \sqleq r$: analogous to the previous case.
\end{enumerate}
  Expansion followed by exchange:
\begin{enumerate}[(i)]
\item $u \sqleq v^{-1} w_{1} w_{2} x^{-1} y^{-1}$ and $w_{2} x^{-1} y^{-1} \sqleq p^{-1} q$: in that case $v u y x \sqleq w_{1} w_{2}$ and $p w_{2} \sqleq q y x$. Depending on how the latter inequality is witnessed: (a) if $p = p_{1} p_{2}$ with $p_{1} \sqleq q$ and $p_{2} w_{2} \sqleq yx$, then $v u p_{2} w_{2} \sqleq vuyx \sqleq w_{1} w_{2}$ so by cancellativity $v u p_{2} \sqleq w_{1}$ and $u \sqleq v^{-1} w_{1} p_{2}^{-1} \sqleq v^{-1} w_{1} p_{2}^{-1} p_{1}^{-1} q = v^{-1} w_{1} p^{-1} q$ (expansion). On the other hand, (b) if $w_{2} = w_{21} w_{22}$ with $p w_{21} \sqleq q$ and $w_{22} \sqleq yx$, then $v u w_{22} \sqleq vuyx \sqleq w_{1} w_{2} = w_{1} w_{21} w_{22}$, so by cancellativity $v u \sqleq w_{1} w_{21}$ and $u \sqleq v^{-1} w_{1} w_{21} \sqleq v^{-1} w_{1} p^{-1} q$ (expansion).
\item $u \sqleq y^{-1} v^{-1} w_{1} w_{2} x^{-1}$ and $y^{-1} v^{-1} w_{1} \sqleq p q^{-1}$: analogous to the previous case.
\item $u \sqleq v^{-1} w_{1} w_{2} x_{2}^{-1} x_{1}^{-1}$ and $w_{2} x_{2}^{-1} \sqleq p^{-1} q$: in that case $v u x_{1} x_{2} \sqleq w_{1} w_{2}$ and $p w_{2} \sqleq q x_{2}$. Depending on how the latter inequality is witnessed: (a) if $p = p_{1} p_{2}$ with $p_{1} \sqleq q$ and $p_{2} w_{2} \sqleq x_{2}$, then $v u x_{1} p_{2} w_{2} \sqleq v u x_{1} x_{2} \sqleq w_{1} w_{2}$,  so by cancellativity $v u x_{1} p_{2} \sqleq w_{1}$ and $u \sqleq v^{-1} w_{1} p_{2}^{-1} x_{1}^{-1} \sqleq v^{-1} w_{1} p^{-1} q x_{1}^{-1}$ (expansion). On the other hand, (b) if $w_{2} = w_{21} w_{22}$ with $p w_{21} \sqleq q$ and $w_{22} \sqleq x_{2}$, then $v u x_{1} w_{22} \sqleq v u x_{1} x_{2} = v u x \sqleq w$, so by cancellativity $v u x_{1} \sqleq w_{1} w_{21}$ and $u \sqleq v^{-1} w_{1} w_{21} x_{1}^{-1} \sqleq v^{-1} w_{1} p^{-1} q x_{1}^{-1}$ (expansion).
\item $u \sqleq v_{2}^{-1} v_{1}^{-1} w_{1} w_{2} x^{-1}$ and $v_{1}^{-1} w_{1} \sqleq p q^{-1}$: analogous to the previous case.
\item $u \sqleq v^{-1} w x_{2}^{-1} x_{1}^{-1}$ and $x_{1}^{-1} p^{-1} q \sqleq r s^{-1}$: in that case $v u x_{1} x_{2} \sqleq w$ and $q s \sqleq p x_{1} r$, so $u p^{-1} q \sqleq u x_{1} r s^{-1} \sqleq v^{-1} w x_{2}^{-1} r s^{-1}$ (exchange and expansion)
\item $u \sqleq v_{2}^{-1} v_{1}^{-1} w x^{-1}$ and $p q^{-1} v_{2}^{-1} \sqleq r^{-1} s$: analogous to the previous case.
\item $u \sqleq w x^{-1}$ and $p w x^{-1} y^{-1} \sqleq q^{-1} r$: in that case $u x \sqleq w$ and $q p w \sqleq r y x$, so $q p u x \sqleq q p w \sqleq r y x$, so by cancellativity $q p u \sqleq r y$ and $p u y^{-1} \sqleq q^{-1} r$ (permutation).
\item $u \sqleq v^{-1} w$ and $x^{-1} v^{-1} w p \sqleq q r^{-1}$: analogous to the previous case.
\end{enumerate}
  Permutation followed by contraction:
\begin{enumerate}[(i)]
\item $u v^{-1} \sqleq x^{-1} y$ and $p^{-1} x^{-1} y q r^{-1} \sqleq s$: in that case $x u \sqleq y v$ and $y q \sqleq x p s r$. Depending on how the former inequality is witnessed: (a) if $x = x_{1} x_{2}$ with $x_{1} \sqleq y$ and $x_{2} u \sqleq v$, then $x_{1} q \sqleq y q \sqleq x_{1} x_{2} p s r$, so $q \sqleq x_{2} p s r$ by cancellativity and $p^{-1} u v^{-1} q r^{-1} \sqleq p^{-1} x_{2}^{-1} q r^{-1} \sqleq s$ (contraction). On the other hand, (b) if $u = u_{1} u_{2}$ with $x u_{1} \sqleq y$ and $u_{2} \sqleq v$, then $x u_{1} q \sqleq y q \sqleq x p s r$, so by cancellativity $u_{1} q \sqleq p s r$ and $p^{-1} u v^{-1} q r^{-1} = p^{-1} u_{1} u_{2} v^{-1} q r^{-1} \sqleq p^{-1} u_{1} q r^{-1} \sqleq s$ (contraction).
\item $u^{-1} v \sqleq x y^{-1}$ and $p^{-1} q x y^{-1} r^{-1} \sqleq s$: analogous to the previous case.
\item $u v^{-1} \sqleq x_{2}^{-1} x_{1}^{-1} y$ and $x_{1}^{-1} y p q^{-1} \sqleq r$: in that case $x_{1} x_{2} u \sqleq y v$ and $y p \sqleq x_{1} r q$. Depending on how the former inequality is witnessed: (a) if $x_{1} = x_{11} x_{12}$ with $x_{11} \sqleq y$ and $x_{12} x_{2} u \sqleq v$, then $x_{11} p \sqleq y p \sqleq x_{1} r q = x_{11} x_{12} r q$, so by cancellativity $p \sqleq x_{12} r q$ and $u v^{-1} p q^{-1} \sqleq x_{2}^{-1} x_{12}^{-1} p q^{-1} \sqleq x_{2}^{-1} r$ (contraction). On the other hand, (b) if $x_{2} = x_{21} x_{22}$ with $x_{1} x_{21} \sqleq y$ and $x_{22} u \sqleq v$, then $x_{1} x_{21} p \sqleq y p \sqleq x_{1} r q$, so by cancellativity $x_{21} p \sqleq r q$ and $u v^{-1} p q^{-1} \sqleq x_{22}^{-1} p q^{-1} \sqleq x_{22}^{-1} x_{21}^{-1} r = x_{2}^{-1} r$ (permutation). Finally, (c) if $u = u_{1} u_{2}$ with $x_{1} x_{2} u_{1} \sqleq y$ and $u_{2} \sqleq v$, then $x_{1} x_{2} u_{1} p \sqleq y p \sqleq x_{1} r q$, so by cancellativity $x_{2} u_{1} p \sqleq r q$ and $u v^{-1} p q^{-1} = u_{1} u_{2} v^{-1} p q^{-1} \sqleq u_{1} p q^{-1} \sqleq u_{1} p q^{-1} \sqleq x_{2}^{-1} r$ (contraction and permutation).
\item $u^{-1} v \sqleq x y_{2}^{-1} y_{1}^{-1}$ and $p^{-1} q x y_{2}^{-1} \sqleq r$: analogous to the previous case.
\item $u v^{-1} \sqleq x_{2}^{-1} x_{1}^{-1} y$ and $p^{-1} q x_{2}^{-1} \sqleq r$: in that case $x_{1} x_{2} u \sqleq y v$ and $q \sqleq p r x_{2}$, so $p^{-1} q u v^{-1} \sqleq r x_{2} u v^{-1} \sqleq r x_{1}^{-1} y$ (contraction and permutation).
\item $u^{-1} v \sqleq x y_{2}^{-1} y_{1}^{-1}$ and $y_{1}^{-1} p q^{-1} \sqleq r$: analogous to the previous case. \qedhere
\end{enumerate}
\end{proof}

\begin{theorem}
  The pomonoid $\FreeUMon{\alg{M}}$ embeds into the pogroup $\FreeRedGroup{\alg{M}}$ if and only if $\alg{M}$ is integrally closed.
\end{theorem}

\begin{proof}
  If $u$ and $v$ are words in $\FreeUMon{\alg{M}}$ such that $u \sqleq v$ in $\FreeRedGroup{\alg{M}}$, then there is a normal proof witnessing this. Since $u$ and $v$ are positive words, this proof cannot contain any rule other than positive monotonicity. In other words, this proof shows that $u \sqleq v$ in $\FreeUMon{\alg{M}}$. Conversely, if $\FreeUMon{\alg{M}}$ embeds into a pogroup, then it is cancellative, and therefore its nuclear image $\alg{M}$ is integrally closed.
\end{proof}

\begin{theorem}
  Integrally closed (integral) pomonoids are precisely the nuclear images of subpomonoids of (negative cones of) pogroups.
\end{theorem}

  If a commutative cancellative pomonoid is residuated, not only does it embed into its Abelian pogroup of fractions, but it is in fact isomorphic to a conuclear image of this pogroup with respect to the map $a^{-1} b \mapsto a \bs b$. In our particular case, this extends beyond the commutative case. However, one has to take into account that a generic element of $\FreeRedGroup{\alg{M}}$ has a more general form and therefore it can be parsed in different ways as a product of residuals.

 To state the appropriate theorem, we first need to introduce the appropriate generalization of conuclear images. We say that a subposet $Q$ of a poset $P$ is an \emph{ideal} subposet if the intersection of $Q$ with each non-empty finitely generated downset of $P$, or equivalently with each principal downset of $P$, is a non-empty finitely generated downset of $Q$. In other words, there is monotone map $\sigma_{Q}\colon P \to \Id P$ defined by $\sigma_{Q}(a) \assign \below(\below a \cap Q)$. We call such a map an \emph{ideal interior operator} and we call $Q$ the \emph{image} of this ideal interior operator. If the restriction of each principal downset of $P$ to $Q$ is a principal downset, then $\sigma_{Q}$ is in fact an interior operator.

  Equivalently, an ideal interior operator on $P$ may be defined as a monotone map $\sigma\colon P \to \Id P$ such that $\sigma^{\sharp}(\sigma(a)) = \sigma(a) \subseteq \below a$ for each $a \in P$, where $\sigma^{\sharp}\colon \Id P \to \Id P$ is the map
\begin{align*}
  f^{\sharp}( \below \{ x_{1}, \dots, x_{m} \}) & \assign \bigcup \set{f(x_{i})}{1 \leq i \leq m}.
\end{align*}
  Yet another possibility would be to define an ideal conucleus on $P$ as a conucleus on $\Id P$ which commutes with binary joins. Each ideal interior operator $\sigma$ on $P$ determines an ideal subposet $P_{\sigma}$, namely the subposet consisting of the \emph{$\sigma$-open} elements of $P$, i.e.\ those $a \in P$ such that $\below a = \sigma(a)$. Equivalently, $a \in P_{\sigma}$ if and only if $a$ is one of the finitely many maximal elements of $\sigma(b)$ for some $b \in P$.

  A subpomonoid of a pomonoid $\alg{M}$ which is also an ideal subposet will called an ideal subpomonoid. An \emph{ideal conucleus} on a pomonoid $\alg{M}$ is then an ideal interior operator $\sigma$ on~$\alg{M}$ such that for each $a, b \in \alg{M}$
\begin{align*}
  \sigma(a) \ast \sigma(b) & \subseteq \sigma(a \cdot b) & & \text{and} & & \sigma(\1) = \1.
\end{align*}
  The image of $\sigma$ (in the sense explained in the previous paragraph) is in fact a subpomonoid of $\alg{M}$ denoted $\alg{M}_{\sigma}$. We call $\alg{M}_{\sigma}$ an \emph{ideal conuclear image} of~$\alg{M}$.

\begin{fact}
  The map $\sigmaideal\colon \Id \alg{M} \to \Id \Id \alg{M}$ such that
\begin{align*}
  \sigmaideal(\below \{ x_{1}, \dots, x_{n} \}) & \assign \below \{ \below x_{1}, \dots, \below x_{n} \}
\end{align*}
  is an ideal conucleus on $\Id \alg{M}$ and $\alg{M} \iso (\Id \alg{M})_{\sigma}$ via the map $a \mapsto \below a$.
\end{fact}

\begin{fact}
  Each ideal subpomonoid (in particular, each ideal conuclear image) of an ideally residuated pomonoid $\alg{M}$ is an ideally residuated subpomonoid of~$\alg{M}$.
\end{fact}

\begin{theorem}
  Let $\alg{M}$ be an ideally residuated integrally closed pomonoid. Then $\FreeUMon{\alg{M}}$ is an ideal conuclear image of $\FreeRedGroup{\alg{M}}$.
\end{theorem}

\begin{proof}
  Consider some $\alpha \in \FreeGroup{\alg{M}}$. We show that there are finitely many positive words $u_{1}, \dots, u_{k}$ such that for each positive word $u$ we have $u \sqleq \alpha$ in $\FreeGroup{\alg{M}}$ if and only if $u \sqleq u_{i}$ for in $\FreeUMon{\alg{M}}$ for some $u_{i}$. It will immediately follow that the map $\sigma\colon \alpha \mapsto \below \{ u_{1}, \dots, u_{k} \}$ is an ideal conucleus: if $u \in \sigma(\alpha)$ and $v \in \sigma(\beta)$, i.e. $u \sqleq \alpha$ and $v \sqleq \beta$ in $\FreeGroup{\alg{M}}$, then $u \circ v \sqleq \alpha \circ \beta$, i.e.\ $u \circ v \in \sigma(\alpha \circ \beta)$.

  Let $u$ be a positive word such that $u \sqleq \alpha$. Then there is a normal proof witnessing this. Because $u$ is a positive word and the proof is normal, the proof does not contain any instances of negative monotonicity, contraction, and permutation. It only contains some instances of expansion, followed by some instances of positive monotonicity. We may assume that different instances of positive monotonicity are independent in the sense that below an instance $p \sqleq [a]$ positive monotonicity is never applied to any letter within this instance of $p$. (Otherwise we may collapse the two instances into one.) Let us track each such instance of $p$ through the proof. There are four options.

  (i) No subword of $p$ is a result of expansion. We then remove the monotonicity rule and replace $u$ by $\beta p \gamma$ by $u' \assign \beta [a] \gamma$. Clearly $u \sqleq u' \sqleq \alpha$.

  (ii) The word $p$ appears as a result of an expansion of the form $x \sqleq v^{-1} w_{1} p w_{2} y^{-1}$. We then replace this expansion by $x \sqleq v^{-1} w_{1} [a] w_{2} y^{-1}$ and remove the corresponding instance of the monotonicity rule.

  (iii) The word $p$ appears as a result of an expansion of the form $p_{1} x \sqleq p_{1} p_{2} v w^{-1}$, where $p = p_{1} p_{2}$. We then replace this expansion by $p_{1} x \sqleq [a] v w^{-1}$ and remove the monotonicity rule.

  (iv) The word $p$ appears as a result of an expansion of the form $x p_{2} \sqleq v^{-1} w p_{1} p_{2}$, where $p = p_{1} p_{2}$. We then replace this expansion by $x p_{2} \sqleq v^{-1} w [a]$ and remove the monotonicity rule. We thus obtain a proof of $u' \sqleq \alpha$ for some $u'$ such that $u \sqleq u'$ which only consists of instances of expansion. Let us call such proofs \emph{expansive}.

  We now define the \emph{rank} of a word in $\FreeGroup{\alg{M}}$ as the number of negative letters in it. Let $n$ be the rank of $\alpha$. We prove that there are finitely many words $\alpha_{1}, \dots, \alpha_{m}$ of rank at most $n-1$ such that if $u \sqleq \alpha$ has an expansive proof, then $u \sqleq \alpha_{i}$ for some $\alpha_{i}$. Applying this claim $n$ times will yield the required finite set of positive words $u_{1}, \dots, u_{n}$.

  The claim holds because within $\alpha$ there are only finitely many possible choices for the right-hand side $u^{-1} v w^{-1}$ of an instance of expansion. Moreover, recall that $\FreeUMon{\alg{M}}$ is ideally residuated because $\alg{M}$ is. For each such subword $u^{-1} v w^{-1}$ of $\alpha$ there are therefore finitely many maximal solutions $x$ to the inequality $x \sqleq u^{-1} v w^{-1}$. This yields our finite set of words $\alpha_{1}, \dots, \alpha_{m}$. Because either $u$ or $w$ is non-empty (by the definition of expansion), the rank of these is at most $n-1$.
\end{proof}

  In the following theorem, we call an ideal conucleus $\sigma$ \emph{negative} if $\sigma(x) \subseteq \below \1$.

\begin{theorem}
  Integrally closed (integral) ideally residuated pomonoids are precisely the nuclear images of (negative) ideal conuclear images of pogroups.
\end{theorem}

\begin{theorem}
  Integral residuated pomonoids are precisely the nuclear images of negative ideal conuclear images of pogroups w.r.t.\ an ideal conucleus $\sigma$ such that $\sigma(u^{-1} v)$ and $\sigma(u v^{-1})$ are principal downsets whenever $u$ and $v$ lie in the ideal conuclear image.
\end{theorem}

\begin{proof}
  The condition on $\sigma$ is equivalent to the claim that the ideal conuclear image is residuated (rather than merely ideally residuated).
\end{proof}

  Unlike in the lattice-ordered case, we cannot easily replace negative ideal nuclear images here by ideal conuclear images of negative cones here. This is because the negative cone conucleus $a \mapsto \1 \wedge a$ is not available in the partially ordered setting. An ideal conucleus on the negative cone might therefore not decompose, as it does in the lattice-ordered case, into the negative cone conucleus followed by an ideal conucleus on the negative cone.

\section{Nuclear images of cancellative \texorpdfstring{s$\ell$-monoids}{sl-monoids}}

  Describing the nuclear images of cancellative s$\ell$-monoids turns out to be somewhat more difficult. The non-integral non-commutative case will be left open.

\begin{lemma} \label{lemma: cancellativity of id}
  The s$\ell$-semigroup $\Id \alg{S}$ is cancellative if and only if $\alg{S}$ satisfies the following universal sentences for each $n \geq 1$:
\begin{align*}
  x_1 y \leq x_2 z_2 ~ \& ~ \dots ~ \& ~ x_{n-1} y \leq x_n z_n ~ \& ~ x_n y \leq x_1 z_1 \implies y \leq z_i \text{ for some } z_i, \\
  y x_1 \leq z_2 x_2 ~ \& ~ \dots ~ \& ~ y x_{n-1} \leq z_n x_n ~ \& ~ y x_n \leq z_1 x_1 \implies y \leq z_i \text{ for some } z_i,
\end{align*}
\end{lemma}

\begin{proof}
  Left to right, suppose that the inequalities $x_1 y \leq x_2 z_2$, \dots, $x_n y \leq x_1 z_1$ hold in $\alg{S}$, i.e.\ $\below x_1 \cdot \below y \leq \below x_2 \cdot \below z_2$, \dots, $\below x_n \cdot \below y \leq \below x_1 \cdot \below z_1$ in $\Id \alg{S}$. Taking $a \assign \below x_1 \vee \dots \vee \below x_n$, $b \assign \below y$, and $c \assign \below z_1 \vee \dots \vee \below z_n$ yields $a \cdot b \leq a \cdot c$ in $\Id \alg{S}$, hence $b \leq c$ in $\Id \alg{S}$ by cancellativity. But this means that $\below y \leq \below z_1 \vee \dots \vee \below z_n$ in $\Id \alg{S}$, i.e.\ $y \leq z_i$ in $\alg{S}$ for some $z_i$.

  Right to left, suppose that $a \cdot b \leq a \cdot c$ in $\Id \alg{S}$. We want to prove that $b \leq c$. We may assume without loss of generality that $b$ has the form $\below y$ for some $y \in \alg{S}$. Now consider $a = \below x_1 \vee \dots \vee \below x_n$ and $c = \below z_1 \vee \dots \vee \below z_m$ and suppose that $a \cdot \below y \leq a \cdot c$ in $\Id \alg{S}$. Then $\below (x_{i} \cdot y) = \below x_{i} \cdot \below y \leq (\below x_1 \vee \dots \vee \below x_n) \cdot (\below z_1 \vee \dots \vee \below z_m)$. It follows that for each $x_{i}$ there are $x_{j}$ and $z_{k}$ such that $x_{i} \cdot y \leq x_{j} \cdot z_{k}$ in $\alg{S}$.

  We inductively define $x'_i \in \{ x_1, \dots, x_n \}$ and $z'_j \in \{ z_1, \dots, z_m \}$ as follows. Take $x'_1 \assign x_1$. If $x'_i$ is defined, there are $x_j$ and $z_k$ such that $x'_i y \leq x_j z_k$. Let $x'_{i+1} \assign x_j$ and $z'_{i+1} \assign z_k$. Eventually, we reach a cycle, i.e.\ $x'_{i+1} = x'_{j}$ for some $j \leq i$. Then we apply the implication assumed in~$\alg{S}$ to obtain that $y \leq z'_i$ in $\alg{S}$ for some $z'_{i}$, therefore $b = \below y \leq \below z'_{i} \leq c$ in $\Id \alg{S}$.
\end{proof}

\begin{proposition}
  $\Id \FreeUMon{\alg{M}}$ is cancellative if $\alg{M}$ is an integral s$\ell$-monoid.
\end{proposition}

\begin{proof}
  We need to verify that the above two implications hold in $\FreeUMon{\alg{M}}$.

  Suppose that $x_i \circ y \sqleq x_{i+1} \circ z_{i+1}$ for each $x_i$, taking $n + 1 = 1$. Then there are decompositions $x_i = p_i \circ q_i$ and $y = u_i \circ v_i$ such that either $q_{i} = \emptyword$ or $u_{i} = \emptyword$ and moreover $p_i \circ u_i \sqleq p_{i+1} \circ q_{i+1}$ and $q_i \circ v_i \sqleq z_{i+1}$. Recalling that $w \circ \emptyword = w \circ [\1]$ and $\emptyword \circ w = [\1] \circ w$ in $\FreeUMon{\alg{M}}$, we may equivalently say that either $q_{i} = [\1]$ or $u_{i} = [\1]$. It now suffices to show that $u_{i} \sqleq q_{i}$ for some $i$, since this implies that $y = u_{i} \circ v_{i} \sqleq q_{i} \circ v_{i} \sqleq z_{i}$. In other words, we need to prove that if $p_{i} \circ u_{i} \sqleq p_{i+1} \circ q_{i+1}$ and either $u_{i} = [\1]$ or $q_{i} = [\1]$ for each $i$, then $u_{i} \sqleq q_{i}$ for some~$i$. Because $\alg{M}$ is integral, this is trivially true if $q_{i} = [\1]$ for some $q_{i}$. Otherwise, $u_{i} = [\1]$ for each $u_{i}$, so $p_{i} \sqleq p_{i+1} q_{i+1}$ for each $i$. Combining these inequalities yields $p_{n} \sqleq p_{n} \circ q_{n} \circ \ldots \circ q_{1}$, therefore by cancellativity $[\1] \sqleq q_{n} \circ \ldots \circ q_{1}$ and $[\1] \sqleq q_{i}$ for each $q_{i}$. By integrality, it follows that $u_{i} \sqleq [\1] = q_{i}$ for each~$q_{i}$.
\end{proof}

\begin{proposition}
  $\Id \FreeSem{\alg{S}}$ is cancellative if $\alg{S}$ is an integral s$\ell$-semigroup.
\end{proposition}

\begin{proof}
  Essentially the same argument goes through if instead of taking $q_{i} = [\1]$ or $u_{i} = [\1]$ we directly remove $q_{i}$ or $u_{i}$ from the relevant inequalities. In the last step, integrality yields $p_{n} \circ v_{i} \sqleq p_{n} \circ q_{n} \circ \ldots \circ q_{1} \circ v_{i} \sqleq q_{i} \circ v_{i}$, so $v_{i} \sqleq q_{i} \circ v_{i}$ by cancellativity.
\end{proof}

\begin{theorem} \label{thm: integral sl-monoids}
  Integral s$\ell$-monoids (s$\ell$-semigroups) are precisely the nuclear images of [distributive] cancellative integral s$\ell$-monoids (s$\ell$-semigroups).
\end{theorem}

  The above argument fails badly in the commutative case. If~$\alg{M}$ is a finite commutative s$\ell$-monoid, $\Id \CFreeRedUMon{\alg{M}}$ is never cancellative unless $\alg{M}$ is trivial.

\begin{proposition}
  Let $\alg{M}$ be a commutative pomonoid. Then $\Id \alg{M}$ is only cancellative if $\alg{M}$ is the trivial (singleton) pomonoid.
\end{proposition}

\begin{proof}
  If $\Id \alg{M}$ is cancellative, then $\alg{M}$ satisfies for each $n \geq 1$ the implication
\begin{align*}
  x_1 y \leq x_2 z_2 ~ \& ~ \dots ~ \& ~ x_{n-1} y \leq x_n z_n ~ \& ~ x_n y \leq x_1 z_1 \implies y \leq z_i \text{ for some } z_i.
\end{align*}
  In particular, $\alg{M}$ is cancellative. The pomonoid $\alg{M}$ must also satisfy the following implication: if $u_{i} v_{i} = u_{i+1} v_{i+1}$ and $p_{i} u_{i} \leq p_{i+1} q_{i+1}$ and $q_{i} v_{i} \leq z_{i+1}$ for each $1 \leq i \leq n$, then $u_{i} v_{i} \leq z_{i}$ for some $1 \leq i \leq n$. (Take $x_{i} = p_{i} q_{i}$ and $y = u_{i} v_{i}$ and apply the previous implication.) Taking $z_{i+1} \assign q_{i} v_{i}$ yields the implication: if $u_{i} v_{i} = u_{i+1} v_{i+1}$ and $p_{i} u_{i} \leq p_{i+1} q_{i+1}$ for each $1 \leq i \leq n$, then $u_{i} v_{i} \leq q_{i} v_{i}$ for some $1 \leq i \leq n$. By cancellativity, this conclusion is equivalent to $u_{i} \leq q_{i}$. Taking $p_{i} = \1$ then yields the implication: $u_{i} \leq q_{i+1}$ for each $1 \leq i \leq n$ implies $u_{i} \leq q_{i}$ for some $1 \leq i \leq n$. Finally, taking $q_{i+1} \assign u_{i}$ yields simply the inequality $u_{i+1} \leq u_{i}$. This holds only if the pomonoid is trivial.
\end{proof}

\newlength{\auxlength}
\newlength{\auxlengthtwo}
\auxlength=\abovedisplayskip
\auxlengthtwo=\belowdisplayskip
\abovedisplayskip=8pt
\belowdisplayskip=8pt

  Nevertheless, we may find a cancellative quotient of $\Id \CFreeRedUMon{\alg{M}}$ whose nuclear image is isomorphic to $\alg{M}$. The key to this will be what we call the \emph{square condition}. Informally speaking, consider a finite set of variables and fill in an $n \times n$ table with these variables so that each variable occurs exactly once in each row. A cell may remain empty or it may contain more than one variable. For example, we may obtain the following table in this way:
\begin{align*}
\begin{tabular}{lll}
  $x_{1}$ & $x_{2}$ & $x_{3} x_{4}$ \\
  & $x_{1} x_{3}$ & $x_{2} x_{4}$ \\
  $x_{4}$ & $x_{1} x_{2} x_{3}$ &  
\end{tabular}
\end{align*}
  The square condition then states that if the product of each column is below $a$, then so is the product of each row. In this case,
\begin{align*}
  x_{1} x_{4} \leq y ~ \& ~ x_{1}^2 x_{2}^{2} x_{3}^2 \leq y~ \& ~ x_{2} x_{3} x_{4}^{2} \leq y & \implies x_{1} x_{2} x_{3} x_{4} \leq y,
\end{align*}
  which is a simple pomonoidal quasi-inequation equivalent to the s$\ell$-monoidal equality
\begin{align*}
  x_{1} x_{2} x_{3} x_{4} \leq x_{1} x_{4} \vee  x_{1}^2 x_{2}^{2} x_{3}^2 \vee x_{2} x_{3} x_{4}^{2}.
\end{align*}
  A simpler example of an inequality obtained in this way is $x y \leq x^{2} \vee y^{2}$.

\abovedisplayskip=\auxlength
\belowdisplayskip\auxlengthtwo

\begin{definition} \label{def: square condition}
  A pomonoid $\alg{M}$ satisfies the \emph{square condition} if the following holds: for each $n$-tuple of decompositions of a product of variables $\pi \assign x_{1} \dots x_{n}$ into $n$ subproducts (allowing for empty subproducts)
\begin{align*}
  \pi = \pi_{11} \pi_{12} \dots \pi_{1n} = \pi_{21} \pi_{22} \dots \pi_{2n} = \dots = \pi_{n1} \pi_{n2} \dots \pi_{nn}
\end{align*}
  the following implication holds: 
\begin{align*}
  \pi_{11} \pi_{21} \dots \pi_{n1} \leq y ~ \& ~ \dots ~ \& ~ \pi_{1n} \pi_{2n} \dots \pi_{nn} \leq y \implies x \leq y.
\end{align*}
\end{definition}

  In an s$\ell$-monoid this implication is equivalent to the inequality
\begin{align*}
  \pi \leq \!\! \displaystyle{\bigvee_{1 \leq i \leq n}} \!\! \pi_{1i} \pi_{2i} \dots \pi_{ni}.
\end{align*}

\begin{fact}
  Commutative cancellative s$\ell$-monoids satisfy the implication
\begin{align*}
  x^{n} \leq y^{n} \implies x \leq y.
\end{align*}
\end{fact}

\begin{proof}
  This follows by cancellativity from the following inequality: $x (x \vee y)^{n-1} = x^{n} \vee x^{n-1} y \vee \dots \vee x y^{n-1} \leq y^{n} \vee x y^{n-1} \vee \dots \vee x^{n-1} y = y (x \vee y)^{n-1}$.
\end{proof}

\begin{fact}
  $\alg{M}$ satisfies the square condition if and only if $\FreeUMon{\alg{M}}$ satisfies for each $n$ the implication
\begin{align*}
  w^{n} \sqleq [a]^{n} \implies w \sqleq [a].
\end{align*}
\end{fact}

\begin{proof}
  The premises of the square condition are precisely what applying the definition of the preorder $\sqleq$ to $w^{n} \sqleq [a]^{n}$ yields. It now suffices to recall that $w \sqleq [a]$ in $\FreeUMon{\alg{M}}$ if and only if $\gamma(w) \leq a$ in $\alg{M}$, and $\gamma(u \circ v) = \gamma(u) \cdot \gamma(v)$.
\end{proof}

  We remind the reader that our definition of a distributive semilattice differs slightly from what they might expect (see p.~\pageref{page: distributive}). In particular, we allow for semilattices which are not down-directed to be distributive, whereas the more common definition implies down-directedness. The two definitions coincide for down-directed semilattices (in particular, for integral s$\ell$-monoids).

\begin{theorem}
  Commutative integrally closed (integral) s$\ell$-monoids satisfying the square condition are precisely the nuclear images of [distributive] commutative cancellative (integral) s$\ell$-monoids.
\end{theorem}

\begin{proof}
  Each commutative cancellative s$\ell$-monoid satisfies the square condition: multiplying the premises yields $x^{n} \leq y^{n}$, therefore $x \leq y$. Because the square condition is a set of simple quasi-inequations, its validity is preserved under nuclear images. Conversely, let $\class{K}$ be the quasivariety of commutative cancellative s$\ell$-monoids and let $\leqK$ be the smallest $\class{K}$-congruence on $\Id \FreeUMon{\alg{M}}$. Then $u \leqK v$ for $u, v \in \Id \FreeUMon{\alg{M}}$ if and only if there is some $w \in \Id \FreeUMon{\alg{M}}$ such that $w \circ u \sqleq w \circ v$. (This holds because the relation $\leqK$ is the smallest cancellative pomonoidal congruence and moreover it is an s$\ell$-monoidal congruence.) It suffices to show that $u \leqK v$ implies $\gamma(u) \leq \gamma(v)$: this implies both that $\leqK$ is an order congruence with respect to the nucleus $[\gamma]$ and that its restriction to the nuclear image is simply the same as the restriction of $\sqleq$.

  The claim that $u \leqK v$ implies $\gamma(u) \leq \gamma(v)$ is equivalent to: $w \circ u \sqleq w \circ v$ and $v \sqleq [a]$ imply $u \sqleq [a]$. This in turn simplifies to: $w \circ u \sqleq w \circ [a]$ implies $u \sqleq [a]$. Let $u = \bigvee_{i \in I} u_{i}$ and $w = \bigvee_{j \in J} w_{j}$, where the elements $u_{i}$ and $w_{j}$ are words in $\FreeUMon{\alg{M}}$. Then the last implication states that $u_{i} \sqleq [a]$ for each $i \in I$ provided that for each $i \in I$ and $j \in J$ there is a $k \in J$ such that $w_{j} \circ u_{i} \sqleq w_{k} \circ [a]$. Let us fix an $i \in I$. We show that $u_{i} \sqleq [a]$ provided that for each $j \in J$ there is $k \in J$ such that $w_{j} \circ u_{i} \sqleq w_{k} \circ [a]$. Among the premises we may find a circle of inequalities $w_{1} \circ u_{i} \sqleq w_{2} \circ [a]$, \dots, $w_{n} \circ u_{i} \sqleq w_{1} \circ [a]$. Multiplying these and applying commutativity yields $w_{1} \circ \ldots \circ w_{n} \circ u_{i}^{n} \sqleq w_{1} \circ \ldots \circ w_{n} \circ [a]^{n}$, so $u_{i}^{n} \sqleq [a]^{n}$ by the cancellativity of $\FreeUMon{\alg{M}}$. The square condition now yields $u_{i} \sqleq [a]$.
\end{proof}

\begin{theorem} \label{thm: commutative integrally closed square sl-monoids}
  Commutative integrally closed (integral) s$\ell$-monoids satisfying the square condition are precisely the nuclear images of sub-s$\ell$-monoids of (negative cones of) Abelian $\ell$-groups.
\end{theorem}

\section{Nuclear images of cancellative RLs}

  Finally, the theorems proved for s$\ell$-monoids in the previous sections immediately lift to residuated lattices, provided that we restrict to dually wpo nuclear images. We only make one additional observation in this section, namely that in the non-commutative case $\Id \FreeUMon{\alg{M}}$ satisfies $x (y \wedge z) = xy \wedge xz$ and $(x \wedge y) z = xz \wedge yz$.

  In the following lemma, $\IdAll \alg{M}$ denotes the lattice of \emph{all} downsets of $\alg{M}$ equipped with the multiplication $X \ast Y \assign \below \set{x \cdot y}{x \in X \text{ and } y \in Y}$. We know already that $\Id \FreeUMon{\alg{M}}$ is a lattice where meets and joins coincide with inter\-sections and unions if $\alg{M}$ is dually wpo. It follows that $\Id \FreeUMon{\alg{M}}$ also satisfies these equations in that case.

\begin{lemma}
  Let $\alg{M}$ be an integral pomonoid. Then $\IdAll \FreeUMon{\alg{M}}$ satisfies the equations $x (y \wedge z) = xy \wedge xz$ and $(x \wedge y) z = xz \wedge yz$.
\end{lemma}

\begin{proof}
  The equation $x (y \wedge z) = xy \wedge xz$ is equivalent in $\IdAll \FreeUMon{\alg{M}}$ to the inequality $xy \wedge xz \leq x (y \wedge z)$, which in turn is equivalent to $x_1 y \wedge x_2 z \leq (x_1 \vee x_2) (y \wedge z)$: in one direction we substitute $x_1 \vee x_2$ for $x$, in other we take $x_1 = x_2$. Because $\IdAll \FreeUMon{\alg{M}}$ is completely distributive and each of its elements is a join of principal downsets, it in fact suffices to prove that the inequality $x_1 y \wedge x_2 z \leq (x_1 \vee x_2) (y \wedge z)$ holds in $\Id \FreeUMon{\alg{M}}$ whenever $x_1, x_2, y, z$ are principal downsets, say $x_1 \assign \below u_1$, $x_2 \assign u_2$, $y \assign \below v$, $z \assign \below w$. For this, it suffices to prove that for each word $t$
\begin{align*}
  \below t \leq \below u_1 \cdot \below v ~ \& ~ \below t \leq \below u_2 \cdot \below w & \implies \below t \leq (\below u_1 \vee \below u_2) \cdot (\below v \wedge \below w).
\end{align*}
  Finally, to prove this implication it suffices to show the following in $\FreeUMon{\alg{M}}$:
\begin{align*}
  t \sqleq u_{1} \circ v ~ \& ~ t \sqleq u_{2} \circ w \implies t \sqleq u_{1} \circ z \text{ or } t \sqleq u_{2} \circ z \text{ for some } z \sqleq v, z \sqleq w.
\end{align*}
  There are decompositions $t = t_{1} \circ t_{2} = t_{3} \circ t_{4}$ such that $t_{1} \sqleq u_{1}$, $t_{2} \sqleq v$, $t_{3} \sqleq u_{2}$, $t_{4} \sqleq w$.\footnote{There is a slight subtlety involved here. For the most part, we have been ignoring the distinction between the preordered structure $\langle M^{+}, \sqleq, \circ, \emptyword \rangle$ and the pomonoid $\FreeUMon{\alg{M}}$. In this case, it matters that the decompositions $t_{1} \circ t_{2}$ and $t_{3} \circ t_{4}$ are decompositions in the preordered structure, i.e.\ in the monoid of words rather than its quotient $\FreeUMon{\alg{M}}$. This is because in the next sentence we take advantage of a special property of this monoid.} Because $t_{1} \circ t_{2} = t_{3} \circ t_{4}$ in the monoid of words over $M$, there is some $y$ such that either $t_{1} \circ y = t_{3}$ and $t_{2} = y \circ t_{4}$ or $t_{1} = t_{3} \circ y$ and $y \circ t_{2} = t_{4}$. Without loss of generality, let us consider the former option. Then by integrality $t_{2} \sqleq t_{4} \sqleq w$, so we may take $z \assign t_{2}$, since $t = t_{1} \circ t_{2} \sqleq u_{1} t_{2}$.
\end{proof}

  By an observation of Galatos and Hor\v{c}\'{i}k~\mbox{\cite[Lemma~4.2]{galatos+horcik13}}, an integral pomonoid generated by a dually wpo subset is a dual wpo, hence $\FreeUMon{\alg{M}}$ is dually wpo if $\alg{M}$ is integral and dually wpo. In that case, every downset of $\FreeUMon{\alg{M}}$ is finitely generated, so $\Id \FreeUMon{\alg{M}}$ is in fact the lattice of all non-empty downsets of $\FreeUMon{\alg{M}}$.

  Let us call a distributive lattice which is isomorphic to the lattice of all non-empty downsets of some poset \emph{almost perfect}, by analogy with perfect distributive lattices, which are isomorphic to the lattice of all downsets of some poset. In other words, an almost perfect lattice is one which can be made perfect by appending a bottom element. (The reason why we exclude lattices with a bottom element is that this is of course incompatible with cancellativity.)

\begin{theorem} \label{thm: finite integral cancellative rls}
  The following classes of residuated lattices (RLs) coincide:
\begin{enumerate}[(i)]
\item dually wpo integral RLs,
\item dually wpo nuclear images of cancellative RLs,
\item dually wpo nuclear images of integral cancellative RLs,
\item dually wpo nuclear images of (almost perfect) distributive integral cancellative RLs satisfying the equations $x (y \wedge z) = xy \wedge xz$ and $(x \wedge y) z = xz \wedge yz$.
\end{enumerate}
\end{theorem}

  The last item of this theorem suggests that $\ell$-groups may also be implicated in this equivalence. These additional conditions (distributivity and the two equations) are satisfied by every conuclear image of an $\ell$-group with respect to a conucleus $\sigma$ such that $\sigma(x \wedge y) = \sigma(x) \wedge \sigma(y)$. This raises the following question: is each dually wpo integral residuated lattice the nuclear image of a negative conuclear image of an $\ell$-group with respect to a conucleus $\sigma$ such that $\sigma(x \wedge y) = \sigma(x) \wedge \sigma(y)$?

  While we do not know whether each integral residuated lattice is a nuclear image of an integral cancellative one, the finite embeddability property implies that at least each integral residuated lattice embeds into such a nuclear image.

\begin{theorem}
  Each integral residuated lattice embeds into a nuclear image of some distributive integral cancellative residuated lattice which satisfies the equations $x (y \wedge z) = xy \wedge xz$ and $(x \wedge y) z = xz \wedge yz$.
\end{theorem}

\begin{proof}
  A result of Blok and van Alten~\cite{blok+alten05} states that the variety of integral residuated lattices is generated as a universal class by its finite algebras. In other words, each integral residuated residuated lattice embeds into an ultraproduct of finite integral residuated lattices. But the nuclear image of an ultraproduct of nuclear residuated lattices is isomorphic to the corresponding ultraproduct of nuclear images of these nuclear residuated lattices.
\end{proof}

  In the following theorem, by a \emph{negative conuclear image} we mean a conuclear image with respect to some $\sigma$ such that $\sigma(x) \leq \1$. Instead of talking about negative conuclear images, we may equivalently talk about conuclear images of negative cones of Abelian $\ell$-groups.

\begin{theorem} \label{thm: finite integral commutative square rls}
  The following classes of commutative residuated lattices (CRLs) coincide:
\begin{enumerate}[(i)]
\item dually wpo integral CRLs satisfying the square condition,
\item dually wpo nuclear images of cancellative CRLs,
\item dually wpo nuclear images of integral cancellative CRLs,
\item dually wpo nuclear images of (negative) conuclear images of Abelian $\ell$-groups.
\end{enumerate}
\end{theorem}

  One might wish for an analogous theorem for integrally closed residuated lattices but each dually wpo residuated lattice has a top element and each integrally closed pomonoid with a top element is in fact integral. To obtain a theorem that is genuinely about integrally closed residuated lattices rather than integral ones, one would need to relax the condition of being dually wpo. Inspecting the proof of the above theorem shows that we do not in fact require the entire free nuclear preimage to be dually wpo. It would suffice for each principal downset of $\FreeUMon{\alg{M}}$ to be dually wpo. However, we are not aware of a suitable sufficient codition on $\alg{M}$ which would ensure this. The problem lies in ensuring that there is no infinite antichain of words $[a,b]$ such that $[a] \circ [b] \sqleq [c]$, i.e.\ $a \cdot b \leq c$ with forcing $\alg{M}$ itself to be dually wpo.

  Again, while we do not know whether each integral commutative residuated lattice is a nuclear image of an integral cancellative one, the finite embeddability property yields the following weaker result.

\begin{theorem}
  Each integral commutative residuated lattice embeds into a nuclear image of a distributive integral commutative cancellative residuated lattice.
\end{theorem}

\begin{proof}The variety of integral commutative residuated lattices is generated as a universal class by its finite algebras~\cite{blok+alten05}.
\end{proof}

  Let us end the paper with an open problem. We have managed to describe the finite nuclear images of integral commutative cancellative residuated lattices. One might further wish to describe the unit intervals of such residuated lattices (i.e.\ nuclear images with respect to nuclei of the form $\gamma_{u}(x) \assign u \vee x$). The interest of this problem derives from a result of Young~\cite{young14}, who proved in an unpublished manuscript that the unit intervals of integral commutative cancellative residuated lattices (i.e.\ of conuclear images of negative cones of Abelian $\ell$-groups) are precisely the conuclear images of MV-algebras (i.e.\ of unit intervals of negative cones of Abelian $\ell$-groups). Given that Heyting algebras are precisely the conuclear images of Boolean algebras by the classical result of McKinsey and Tarski~\cite{mckinsey+tarski46}, a solution of this problem would yield a class of residuated lattices which are related to MV-algebras as Heyting algebras are to Boolean algebras.

\section*{Acknowledgements}

  The author would like to thank Constantine Tsinakis for his valuable feedback on earlier versions of this manuscript and for posing the (still open) question from which this research stemmed, namely how to describe the unit intervals of commutative cancellative residuated lattices. The author is also grateful to the anonymous referee for a thorough reading of the manuscript and helpful suggestions for improving it, and to Thomas Vetterlein for drawing his attention to the papers~\cite{dvurecenskij+vetterlein01a,dvurecenskij+vetterlein01b}.

\end{document}